\UseRawInputEncoding
\documentclass[reqno]{amsart}
\usepackage{amssymb,latexsym}
\usepackage{amsmath,color}
\usepackage{amsthm}
\usepackage{mathrsfs}
\usepackage{srcltx}
\usepackage{epsfig}
\usepackage{amsmath,amssymb,amsthm}
\usepackage{graphicx}
\usepackage{hyperref}
\usepackage{titletoc}
\usepackage{refcheck}
\usepackage{palatino,mathpazo}
\usepackage[all]{xy}
\usepackage{newlfont}

\usepackage{enumerate}
\numberwithin{equation}{section}
\newtheorem{theorem}{Theorem}[section]
\newtheorem{proposition}[theorem]{Proposition}
\newtheorem{lemma}[theorem]{Lemma}
\newtheorem{corollary}[theorem]{Corollary}

\newtheorem{definition}{Definition}[section]

\numberwithin{equation}{section}

\newtheorem{remark}{Remark}[section]

\def\XXint#1#2#3{{\setbox0=\hbox{$#1{#2#3}{\int}$}
		\vcenter{\hbox{$#2#3$}}\kern-.5\wd0}}

\numberwithin{equation}{section}

\newcommand{\ba}{\begin{array}}
	\newcommand{\ea}{\end{array}}

\begin{document}

\title[steady compactly supported Euler flows with constant vorticity]{\Large On two-dimensional steady compactly supported Euler flows with constant vorticity}
	
\author[C.F. Gui]{Changfeng Gui}
\address{Changfeng Gui, Department of Mathematics, Faculty of Science and Technology, University of Macau, Macau, P.R. China}
\email{changfenggui@um.edu.mo}

\author[J. Wang]{Jun Wang}
\address{Jun Wang, School of Mathematical Sciences, Jiangsu University, Zhenjiang, Jiangsu, 212013, P.R. China}
\email{wangmath2011@126.com}

\author[W. Yang]{Wen Yang}
\address{\noindent Wen ~Yang,~Department of Mathematics, Faculty of Science and Technology, University of Macau, Macau, P.R. China}
\email{wyang@um.edu.mo}
	
\author[Y. Zhang]{Yong Zhang}
\address{Yong Zhang, School of Mathematical Sciences, Jiangsu University, Zhenjiang, Jiangsu, 212013, P.R. China}
\email{zhangyong@ujs.edu.cn}

\begin{abstract}
In this paper, we study the two-dimensional steady compactly supported incompressible Euler equations with free boundaries. We consider flows with constant vorticity that are perturbations of annular or disk equilibria, in contrast to the laminar flows that predominate in the existing literature on steady water waves. More precisely, we analyze three distinct classes of steady Euler flows with compact support, which correspond, respectively, to partially overdetermined problem, two-phase overdetermined problem, and two-phase elliptic overdetermined problem with surface tension. Our main contributions are threefold. For each class, we first prove a flexibility result, that is the existence of nontrivial admissible domains, by combining shape derivatives with local bifurcation theory. Second, we establish the corresponding rigidity results. Third, we apply the implicit function theorem to show that the standard annular flows are stable under small perturbations of the Neumann boundary condition. These results provide new perspectives on the theory of overdetermined elliptic problems.
\end{abstract}

\maketitle
\medskip

\noindent \emph{AMS Subjection Classification(2010):} 35B32; 35N05; 51M10

\section{Introduction and main results}

In this paper, we consider the two-dimensional steady incompressible Euler equations with vorticity, which are known as the water waves problem \cite[(2.1)--(2.2)]{ConstantinSV}. In this setting, the fluid domain \(\Omega\) is bounded by a fixed boundary \(\mathcal{B}\) and a free, sufficiently smooth, non-self-intersecting curve
\[
\mathcal{S}(s) := \bigl\{ (x(s), y(s)) : s \in \mathbb{R} \bigr\}.
\]
One seeks a smooth stream function \(\psi(x,y)\) defined on \(\Omega\) whose reverse gradient gives the velocity field \((\psi_y, -\psi_x)\), satisfying the system
\begin{equation}\label{e1.1}
\begin{cases}
\Delta \psi = \gamma(\psi) & \text{in } \Omega, \\
\psi = 0 & \text{on } \mathcal{S}, \\
\psi = m & \text{on } \mathcal{B}, \\
|\nabla \psi|^2 + 2gy = Q & \text{on } \mathcal{S},
\end{cases}
\end{equation}
where \(\gamma(\psi)\) is the vorticity function, \(g\) denotes the gravitational acceleration, and the constants \(m\) and \(Q\) represent the relative mass flux and the total head, respectively. Although system \eqref{e1.1} with non-closed streamlines \(\mathcal{B}\) and \(\mathcal{S}\), referred to as the steady water waves problem, has been widely studied \cite{ConstantinS, ConstantinSV, ConstantinV, DaiZ, Varholm, Wahlen, WahlenW2}, the closed-streamline configuration has received considerably less attention. Our work aims to address this gap by constructing local curves of steady solutions in the closed-streamline case. To sharpen the focus, we restrict attention to steady periodic solutions throughout the introduction.

The mathematical study of steady water waves dates back more than two centuries to the pioneering contributions of Laplace and Lagrange in the late eighteenth century. Early investigations by Cauchy, Poisson, and Airy typically relied on linearizations near still water. A landmark exact solution was obtained by Gerstner, who derived a family of fully nonlinear steady waves. Stokes \cite{Stokes} advanced the subject significantly by employing perturbation expansions in wave amplitude to analyse travelling waves and by proposing the celebrated conjecture on the existence of a wave of greatest height, featuring sharp crests with an included angle of \(2\pi/3\).

Throughout the 20th century, research focused primarily on two-dimensional irrotational flows, in which the stream function \(\psi\) is harmonic. Reformulating the problem via the Green's function of the Laplacian led to an integral equation on the free boundary \(\mathcal{S}\). Nekrasov \cite{Nekrasov} pioneered this approach, constructing the first rigorous small-amplitude periodic solutions over infinite depth by expanding in powers of the amplitude and establishing a positive radius of convergence. Levi-Civita \cite{LeviC} independently obtained similar results via a direct power-series method. All these early works concerned small-amplitude waves. A major breakthrough occurred with Keady and Norbury \cite{KeadyN}, who applied global bifurcation theory to obtain a continuous branch of large-amplitude solutions. Amick, Fraenkel, and Toland \cite{AmickFT} subsequently analysed the limiting behaviour along this branch and rigorously established the existence of the Stokes extreme wave.

At the turn of the century, researchers began to seriously address the effects of vorticity. The inclusion of nontrivial vorticity presented a major challenge: it was unclear how to construct solutions that went beyond small perturbations of a flat surface. Constantin and Strauss \cite{ConstantinS} treat this difficulty by adapting the semi-hodograph transformation and bifurcation theory, thereby proving the existence of large-amplitude water waves with arbitrary smooth vorticity distributions. However, their method excludes overhanging profiles, stagnation points, and critical layers. Subsequent studies addressed these limitations by employing flattening transformations that rescale the vertical coordinate along each vertical ray to a fixed interval (see, e.g., \cite{DaiXZ, HenryM, Varholm, Wahlen}). A principal drawback of this approach is that it requires the wave profile to be a graph, thereby still excluding overhanging waves. To remove this restriction, Constantin et al.\ \cite{ ConstantinV,ConstantinSV} introduced a conformal map technique that represents the fluid domain as the image of a horizontal strip. This framework permits the construction of both small- and large-amplitude waves with constant vorticity without any a priori geometric assumptions on the free surface or stream function, and it accommodates critical layers, stagnation points, and overhanging profiles. The versatility of this approach has been demonstrated through extensions to stratified waves \cite{Haziot}, capillary-gravity waves \cite{WangXZ}, and electrohydrodynamic waves \cite{DaiFZ}. More recently, Wahl\'en and Weber \cite{WahlenW1, WahlenW2} employed a conformal change of variables and bifurcation theory to establish the existence of large-amplitude capillary-gravity and gravity waves with a prescribed but arbitrary vorticity distribution, again allowing stagnation points, critical layers, and overhanging profiles; a key innovation in their work is the removal of any structural assumptions on the vorticity.

All the steady solutions described above are confined to strip-shaped domains. In many physical situations, however, the free surface may be closed, as in liquid drops or vortex patch. It is therefore natural to investigate the steady water wave problem \eqref{e1.1} in bounded domains. This motivates the study of Euler flows possessing both a closed free surface and a closed bed. Here we restrict attention here to constant vorticity \(\gamma(\psi) = \gamma\), as in \cite{ConstantinSV,ConstantinV}. More precisely, we consider the system
\begin{equation}
\label{e1.2}
\begin{cases}
\Delta \psi = \gamma & \text{in } \Omega \setminus \overline{D_0}, \\
\psi = 0 & \text{on } \partial \Omega, \\
\psi = 1 & \text{on } \partial D_0, \\
|\nabla \psi|^2 = Q > 0 & \text{on } \partial \Omega
\end{cases}
\end{equation}
on a ring-shaped domain \(\Omega \setminus D_0 \subset \mathbb{R}^2\), where \(D_0 = B_\lambda\) is a circle of radius \(\lambda\), \(\partial \Omega\) is a free boundary, \(\psi\) is the stream function, \(\gamma\) is the constant vorticity, and \(Q\) is the Bernoulli constant. In this configuration all streamlines lie at the same horizontal level, so the gravitational potential term \(gy\) in \eqref{e1.1} may be set to zero. System \eqref{e1.2} is known as the Bernoulli free-boundary problem \cite[(1.1)]{HenrotO}.

The second model extends recent work \cite{GuiWYZ} on steady water waves with piecewise smooth vorticity to the closed-streamline setting. We consider
\begin{equation}\label{e1.3}
\begin{cases}
\Delta \psi_1 = \gamma_1 & \text{in } D, \\
\Delta \psi_2 = \gamma_2 & \text{in } \Omega \setminus \overline{D}, \\
\psi_1 = \psi_2 & \text{on } \partial D, \\
\frac{1}{\gamma_1} \partial_\nu \psi_1 = \frac{1}{\gamma_2} \partial_\nu \psi_2 & \text{on } \partial D, \\
\psi_2 = 0 & \text{on } \partial \Omega, \\
|\nabla \psi_2|^2 = Q > 0 & \text{on } \partial \Omega,
\end{cases}
\end{equation}
on a disk-shaped domain \(\Omega \subset \mathbb{R}^2\), where \(D \subset \Omega\) is the core region, \(\partial D\) and \(\partial \Omega\) are free boundaries, \(\psi_1\) and \(\psi_2\) are the stream functions on \(D\) and \(\Omega \setminus \overline{D}\) respectively, \(\gamma_1\) and \(\gamma_2\) are two constant vorticities, \(\nu\) is the outward unit normal to \(\partial D\), and \(Q\) is the Bernoulli constant.

Finally, we consider an additional two-phase elliptic overdetermined problem with surface tension arising in the modeling of liquid drops (see \cite{Meyer}):
\begin{equation}\label{e1.4}
\begin{cases}
\Delta \widetilde{\psi}_1 = \widetilde{\gamma} & \text{in } \Omega, \\
\Delta \widetilde{\psi}_2 = 0 & \text{in } \mathbb{R}^2 \setminus \overline{\Omega}, \\
\widetilde{\psi}_1 = \widetilde{\psi}_2=0 & \text{on } \partial \Omega, \\
| \nabla \widetilde{\psi}_1|^2+\beta\mathcal{K} -q| \nabla \widetilde{\psi}_2|^2 =Q & \text{on } \partial \Omega, \\
|\nabla \widetilde{\psi}_2|\rightarrow 0 & \text{for}~ |x|\rightarrow+\infty,
\end{cases}
\end{equation}
where \(\widetilde{\psi}_1\) and \(\widetilde{\psi}_2\) are the stream functions on \(\Omega\) and \(\mathbb{R}^2 \setminus \overline{\Omega}\) respectively, \(\widetilde{\gamma}\) is the constant vorticity, \(\partial\Omega\) is a free interface near $\partial B_1$, constant $q>0$ represents the ratio of the mass densities of the two phases, $Q$ is Bernoulli's constant and $\beta>0$ is the coefficient of surface tension, and $\mathcal{K}$ denotes the curvature of \(\partial \Omega\).

We note that \eqref{e1.2} is a partially overdetermined elliptic problem \cite{FarinaV,FragalG}, and \eqref{e1.3} and \eqref{e1.4} are two-phase overdetermined problems \cite{Cavallina, CavallinaY}. For further constructions of nontrivial domains in the context of overdetermined elliptic problems we refer the reader to \cite{CFM, DaiSZ, DaiZC, DelPW, FallMW, FallMW1, KamburovS, RosRS, Ruiz1, RuizSW, SchlenkS,Serrin, Sicbaldi}. Related applications of overdetermined elliptic problems to the construction of nontrivial Euler flows appear in \cite{dpmw, DEP, EncisoARS, Gomez, HamelN, Ruiz}.

Without loss of generality, we take \(\Omega_\eta\) to be the perturbed domain of the reference disk \(\Omega_0 = B_1\) (a circle of radius 1), where \(\eta \in C^{2,\alpha}(\partial \Omega_0, \mathbb{R})\) for \(0 < \alpha < 1\). When \(\eta\) is sufficiently small, the perturbed domains are well-defined as the unique bounded domains whose boundaries are
$$
\partial \Omega_\eta = \{ x + \eta(x)\nu(x) : x \in \partial \Omega_0 \}
$$
where \(\nu\) denotes the outward unit normal.

In this paper we shall investigate the effect of constant vorticity \(\gamma\) on the \emph{admissible domains} for the three classes of problems \eqref{e1.2}--\eqref{e1.4}, where an admissible domain is one that supports a solution. In particular, we address the following fundamental questions:
\begin{itemize}
\item \textbf{Flexibility}: For which values of \(\gamma\) does problem \eqref{e1.2}, \eqref{e1.3}, or \eqref{e1.4} admit nontrivial admissible domains?
\item \textbf{Rigidity}: Under what conditions must every admissible domain for \eqref{e1.2}, \eqref{e1.3}, or \eqref{e1.4} be a standard annulus (or a disk)?
\item \textbf{Stability}: When rigidity holds, is the corresponding trivial solution stable under small perturbations of the Neumann boundary condition?
\end{itemize}

To answer these questions we reformulate problems \eqref{e1.2}--\eqref{e1.4} as abstract operator equations. This is accomplished by treating the constant vorticity \(\gamma\) as a bifurcation parameter and regarding the boundary perturbation \(\eta\) as the solution themselves.

The main result for problem \eqref{e1.2} is the following theorem.

\begin{theorem}\label{thm1.1}
Fix some \(\lambda \in (0,1)\), there exists \(s_0 > 0\) and a \(C^1\)-curve of solutions \((\gamma(s), \eta(s))\) to \eqref{e1.2}, parametrized by \(s \in (-s_0, s_0)\), satisfying:
\begin{itemize}
\item[(1)] At \(s = 0\), \((\gamma(0), \eta(0)) = (\gamma^*, 0)\), where
\[
\gamma^* = \frac{4}{\lambda^2 - 2\lambda^2 \ln \lambda - 1},
\]
and the admissible domain \(\Omega_{\eta(0)} \setminus B_\lambda\) is the standard annulus \(B_1 \setminus B_\lambda\).

\item[(2)] For all \(s \neq 0\) sufficiently small, the admissible domain \(\Omega_{\eta(s)} \setminus B_\lambda\) is not an annulus. It is given by (see Figure \ref{figure1})
\[
\Omega_{\eta(s)} \setminus B_\lambda = \{ x \in \mathbb{R}^2 : \lambda < |x| < 1 + \eta(s) \},
\]
where the perturbation takes the form
\[
\eta(s) = s \alpha_1 \cos \theta + o(s), \quad \alpha_1 \neq 0, \quad \theta = \frac{x}{|x|}.
\]
\end{itemize}
\end{theorem}

\begin{figure}
\label{figure1}
\centering
\includegraphics[width=0.7\textwidth]{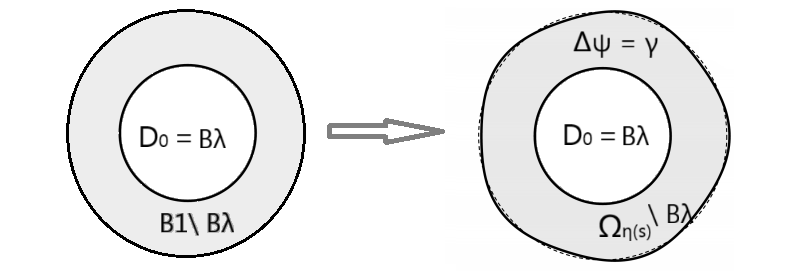}
\caption{The nontrivial domains $\Omega_{\eta(s)}\setminus
B_{\lambda}$ for Theorem \ref{thm1.1} bifurcating from
$B_{1}\setminus B_{\lambda}$.}
\end{figure}

\begin{remark} \label{rem1.2}
(i).
In problem \eqref{e1.2}, the inner disk $B_\lambda$ is fixed throughout the analysis. Hence the problem is not invariant under translations of the whole configuration. Consequently, the kernel element $\cos\theta$ does not arise from an infinitesimal translation, which
follows solely from the spectral relation
$$
\sigma_1(\gamma^*)=0,
\qquad
\sigma_k(\gamma^*)\neq0
\quad (k\neq1),
$$
where \(\sigma_k\) is given in \eqref{e3.8}.
In Theorem \ref{thm1.1}, we only identify a bifurcation point at
\[
\gamma^* = \frac{4}{\lambda^2 - 2\lambda^2 \ln \lambda - 1} \in (-\infty, -4) \quad \mbox{for}\quad \lambda \in (0,1).
\]
In fact, in section 3 we shall show that for every positive integer \(k\), there is a corresponding bifurcation value \(\gamma_k^*\) given by \eqref{e3.9}. When \(\lambda \in (0,0.2483)\), the values \(\gamma_k^*\) for \(k \geq 2\) yield additional nontrivial domains \(\Omega_{\eta(s)}\) with perturbations of the form
\[
\eta(s) = s \alpha_k \cos(k\theta) + o(s).
\]
Moreover,
\[
\gamma_k^* = -\frac{4(k + k\lambda^{2k} + 1 - \lambda^{2k})}{\bigl(1 - \lambda^2 + 2\ln\lambda\bigr)(k + k\lambda^{2k}) + \bigl(1 - \lambda^2 - 2\ln\lambda\bigr)(1 - \lambda^{2k})} > 0
\]
for \(k \in \mathbb{N}^+ \setminus \{1\}\) and \(\lambda \in (0,0.2483)\). Hence positive constant vorticity also permits nontrivial admissible domains for \eqref{e1.2}.
\medskip

\noindent (ii). On the other hand, if \(\gamma \geq 0\) and the solution \(\psi\) of \eqref{e1.2} is positive, then the admissible domain \(\Omega \setminus \overline{B_\lambda}\) must be an annulus, and every solution is radially symmetric and decreasing in \(|x|\). This follows directly from the maximum principle (which forces \(\psi \in (0,1)\) in \(\Omega \setminus \overline{B_\lambda}\)) together with Reichel's theorem (Theorem \ref{thm6.3} in the Appendix).
\medskip

\noindent (iii). In contrast to (ii), any nontrivial solution \(\psi\) of \eqref{e1.2} bifurcating from \(\gamma_k^*\) for \(k \geq 2\) must be sign-changing. Otherwise, if \(\psi > 0\), the maximum principle would again imply \(\psi \in (0,1)\), forcing \(\psi\) to be the trivial radial solution --- a contradiction. Thus positivity of \(\psi\) is essential for the symmetry result. Similar arguments appear in the study of Serrin-type problems \cite{Ruiz1, Wheeler}.
\end{remark}

The rigidity result in Remark \ref{rem1.2}-(ii) shows that, for prescribed nonnegative vorticity \(\gamma\) and positive solutions \(\psi\), problem \eqref{e1.2} admits only radially symmetric solutions in annular domains.
A natural question is whether this radial symmetry persists under small perturbations of the Bernoulli constant \(Q\) (see \cite{GilsbachO, Onodera}). To address this, we consider the perturbed problem
\begin{equation}\label{else}
\begin{cases}
\Delta \psi = \gamma & \text{in } \Omega_\eta \setminus \overline{D_0}, \\
\psi = 0 & \text{on } \partial \Omega_\eta, \\
\psi = 1 & \text{on } \partial D_0, \\
|\nabla \psi|^2 = Q + \rho(\theta) > 0 & \text{on } \partial \Omega_\eta,
\end{cases}
\end{equation}
where \(\rho(\theta)\) is a small perturbation defined on \(\partial \Omega_0 = \partial B_1\).

The following corollary of Theorem \ref{thm1.1} gives an affirmative answer, establishing both existence and local uniqueness, together with a precise asymptotic description of the perturbation.

\begin{corollary}\label{co}
Let $X$ and $Y$ be the zero-mean Banach spaces defined in \eqref{e2.1}.
Fix $\lambda\in(0,1)$ and $\gamma\in\mathbb R$, and assume that
\[ \sigma_k(\gamma)\neq0 \qquad\text{for every }k\in\mathbb N^+, \]
where $\sigma_k(\gamma)$ is given in \eqref{e3.8}.
Let
\[ \bigl(\psi_{\mathrm{tri}}^\gamma,0\bigr) \]
be the trivial solution corresponding to the annulus $B_1\setminus\overline{B_\lambda}$, and let
\[ Q^\gamma = \left( \frac{4+(1-\lambda^2)\gamma}{4\log\lambda} +\frac{\gamma}{2} \right)^2. \]
Then there exist $\varepsilon>0$ and unique $C^1$ maps
\[ \rho\longmapsto \eta(\rho)\in X, \qquad \rho\longmapsto Q(\rho)\in\mathbb R, \]
defined for every $\rho\in Y$ with $\|\rho\|_Y<\varepsilon$,
such that
\[ \eta(0)=0, \qquad Q(0)=Q^\gamma, \]
and the domain
 \[ \Omega_{\eta(\rho)}\setminus\overline{B_\lambda} \]
 admits a solution to \eqref{else}.
Moreover, the pair $\bigl(Q(\rho),\eta(\rho)\bigr)$ is locally unique:
if $\widetilde Q\in\mathbb R$ and $\widetilde\eta\in X$ are sufficiently close to $(Q^\gamma,0)$ and the corresponding solution satisfies
\[ |\nabla\widetilde\psi|^2 = \widetilde Q+\rho \qquad\text{on }\partial\Omega_{\widetilde\eta}, \]
then
\[ \widetilde\eta=\eta(\rho), \qquad \widetilde Q=Q(\rho). \]
If
\[ \rho(\theta) = \sum_{k\geq1}\tau_k\cos(k\theta), \]
then
\[ \eta(\rho) = \sum_{k\geq1} \frac{\tau_k} {2\sqrt{Q^\gamma}\,\sigma_k(\gamma)} \cos(k\theta) + o(\|\rho\|_Y) \]
in $X$ as $\|\rho\|_Y\to0$.
\end{corollary}

We would like to mention that the assumption on \(\gamma\) in Corollary \ref{co} is non-vacuous. In particular, when \(\gamma = 0\),
\[
\sigma_k = \frac{(1-k)\lambda^{2k} - k - 1}{\ln\lambda\,(1 - \lambda^{2k})} > 0.
\]
Moreover, for \(\gamma = 0\) the maximum principle forces any solution \(\psi\) of \eqref{e1.2} to be positive, so \((\psi^\gamma_\mathrm{tri}(|x|), 0)\) is the unique solution (see Remark \ref{rem1.2}-(ii)).

We next turn to the two-phase overdetermined problem \eqref{e1.3}. Here we fix the inner vorticity \(\gamma_1\) and treat the outer vorticity \(\gamma_2\) as the bifurcation parameter, with the boundary perturbation \(\eta \in C^{2,\alpha}(\partial \Omega_0;\mathbb{R})\) as the unknown. The main result is as follows.

\begin{theorem}\label{thm1.2}
Fix some \(\lambda \in (0,1)\) and \(\gamma_1 \in \mathbb{R} \setminus \{0\}\). There exists \(s_0 > 0\) and a \(C^1\)-curve of solutions \((\gamma_2(s), \eta(s))\) to problem \eqref{e1.3}, parametrized by \(s \in (-s_0, s_0)\), with the following properties:
\medskip

\noindent ${\mathrm(1)}$ At \(s = 0\), \((\gamma_2(0), \eta(0)) = (\gamma_1, 0)\) and the admissible domain \(B_\lambda \cup (\Omega_{\eta(0)} \setminus B_\lambda)\) consists of concentric disks.
\medskip

\noindent ${\mathrm(2)}$ For all \(s \neq 0\) sufficiently small, the admissible domain \(B_\lambda \cup (\Omega_{\eta(s)} \setminus B_\lambda)\) is not a pair of concentric disks. It is given by (see Figure 2)
\[
B_\lambda \cup (\Omega_{\eta(s)} \setminus B_\lambda) = \{ x \in \mathbb{R}^2|0 \le |x| < \lambda \} \cup \{ x \in \mathbb{R}^2| \lambda < |x| < 1 + \eta(s) \},
\]
where
\[
\eta(s) = s \alpha_1 \cos\theta + o(s), \quad \alpha_1 \neq 0, \quad \theta = \frac{x}{|x|}.
\]

\noindent ${\mathrm(3)}$ If, in addition, \(\partial_{\nu\nu} \psi_2 = m\) on \(\partial \Omega_\eta\) for some constant \(m\),
then problem \eqref{e1.3} admits a solution only if \(\Omega_\eta\) is a disk.
\end{theorem}

\begin{figure}[ht]
\centering
\includegraphics[width=0.7\textwidth]{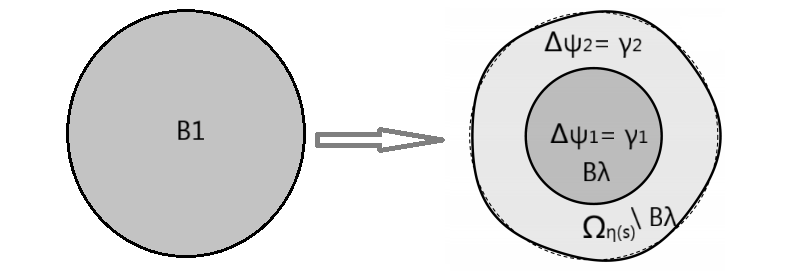}
\caption{The nontrivial domains
$B_{\lambda}\cup\left(\Omega_{\eta(s)}\setminus B_{\lambda}\right)$
for Theorem \ref{thm1.2} bifurcating from $B_1$}.
\end{figure}

\begin{remark}
\noindent (i). In Theorem \ref{thm1.2}, the inner disk $B_\lambda$ is also fixed throughout the analysis. Hence the problem \eqref{e1.3} is not invariant under translations of the whole configuration. Consequently, the kernel element $\cos\theta$ does not arise from an infinitesimal translation. Here we show a bifurcation point at \(\gamma_2^* = \gamma_1\). Thus, when the outer vorticity \(\gamma_2\) coincides with the prescribed inner vorticity \(\gamma_1\), problem \eqref{e1.3} admits nontrivial admissible domains of the form \(B_\lambda \cup (\Omega_{\eta(s)} \setminus B_\lambda)\) bifurcating from the unit disk \(B_1\) (see Figure 2). As shown in Section 4, for every positive integer \(k\) there exists a corresponding bifurcation value \(\gamma_{2k}^*\) given by \eqref{e4.10} (with \(\gamma_{21}^* = \gamma_1\)). For \(\lambda \in (0,1)\), each such \(\gamma_{2k}^*\) generates additional nontrivial domains with perturbations \(\eta(s) = s \alpha_k \cos(k\theta) + o(s)\), bifurcating from the concentric configuration \(B_\lambda \cup (B_1 \setminus B_\lambda)\) (see Figure 3).

\begin{figure}[ht]
\centering
\includegraphics[width=0.7\textwidth]{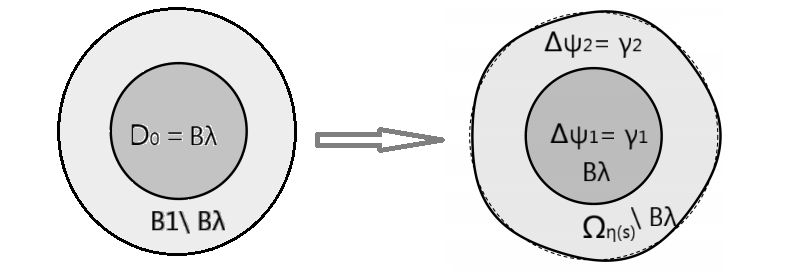}
\caption{The nontrivial domains
$B_{\lambda}\cup\left(\Omega_{\eta(s)}\setminus B_{\lambda}\right)$
bifurcating from $B_{\lambda}\cup\left(B_1\setminus
B_{\lambda}\right)$}.
\end{figure}

\noindent (ii). If \(D_0 = B_\lambda = \emptyset\) and \(\psi_2 > 0\), Serrin's classical result \cite{Serrin} implies that \(\Omega\) must be a disk and \(\psi_2\) is radially symmetric. In our setting, however, \(D_0 = B_\lambda \neq \emptyset\) and two free boundaries are present. It is therefore natural to impose an additional condition on \(\partial \Omega\) to obtain rigidity as in (3). Notably, this rigidity holds without assuming positivity of the solutions.
\end{remark}

Similarly, we establish the stability of the two-phase overdetermined problem \eqref{e1.3}. Consider the perturbed problem
\begin{equation}\label{else1}
\begin{cases}
\Delta \psi_1 = \gamma_1 & \text{in } D_0, \\
\Delta \psi_2 = \gamma_2 & \text{in } \Omega_\eta \setminus \overline{D_0}, \\
\psi_1 = \psi_2 & \text{on } \partial D_0, \\
\frac{1}{\gamma_1} \partial_\nu \psi_1 = \frac{1}{\gamma_2} \partial_\nu \psi_2 & \text{on } \partial D_0, \\
\psi_2 = 0 & \text{on } \partial \Omega_\eta, \\
|\nabla \psi_2|^2 = Q + \rho(\theta) & \text{on } \partial \Omega_\eta,
\end{cases}
\end{equation}
where \(\rho(\theta)\) is a small perturbation defined on
\(\partial \Omega_0 = \partial B_1\).
The following stability result holds.
\begin{corollary}\label{co1}
Assume that
\[ \mu_k\neq0, \qquad k\in\mathbb N^+, \]
where $\mu_k$ is defined in \eqref{e4.9} and let $Q^{\gamma_2} = \frac{\gamma_2^2}{4}$.
Then there exist $\varepsilon>0$ and unique $C^1$ maps
\[ \rho\longmapsto\eta(\rho)\in X, \qquad \rho\longmapsto Q(\rho)\in\mathbb R, \]
defined for every
\[ \rho\in Y, \qquad \|\rho\|_Y<\varepsilon, \]
such that
\[ \eta(0)=0, \qquad Q(0)=Q^{\gamma_2}, \]
and the corresponding transmission problem admits a unique solution
$$\Psi \in C_\mathrm{even}^{2,\alpha}(\Omega_{\eta(\rho)} \setminus \overline{B_\lambda}) \times C_\mathrm{even}^{2,\alpha}(B_\lambda) \cap C_\mathrm{even}^{0,\alpha}(\Omega_{\eta(\rho)})$$
to the perturbed problem \eqref{else1}. Moreover,
\[ \eta(\rho) = \sum_{k\ge1} \frac{\tau_k} {2\sqrt{Q^{\gamma_2}}\,\mu_k} \cos(k\theta) + o(\|\rho\|_Y), \]
provided that
\[ \rho(\theta) = \sum_{k\ge1} \tau_k\cos(k\theta). \]
\end{corollary}

We now turn to the other two-phase overdetermined elliptic problem \eqref{e1.4}. As in the previous cases, we treat the constant vorticity \(\widetilde{\gamma}\) as the bifurcation parameter, while the  boundary perturbation \(\widetilde{\eta} \in C^{2,\alpha}(\partial \Omega_0; \mathbb{R})\) constitutes the unknown. The main result is as follows.

\begin{theorem}\label{thm1.3}
For any given positive coefficient of surface tension $\beta$ and the ratio $q$ of the mass densities of the two phases, there exists \(s_0 > 0\) sufficiently small and a curve of solutions \((\widetilde{\gamma}(s), \widetilde{\eta}(s)\) to \eqref{e1.4}, parametrized by \(s \in (-s_0, s_0)\), with the following properties:
\begin{itemize}
\item[(1)] At \(s = 0\), \((\widetilde{\gamma}(0), \widetilde{\eta}(0)) = (\widetilde{\gamma}_{k,\pm}^*, 0)\),
where $\widetilde{\gamma}_{k,\pm}^*$ is given in \eqref{e5.9} and the admissible domain \(\Omega_{\widetilde{\eta}(0)}\bigcup\mathbb{R}^2 \setminus \Omega_{\widetilde{\eta}(0)} \) is the standard disk and its complement.

\item[(2)] For all \(s \neq 0\) sufficiently small, the admissible domain \(\Omega_{\widetilde{\eta}(s)}\bigcup\mathbb{R}^2 \setminus \Omega_{\widetilde{\eta}(s)} \) is not a disk and its complement. It is given by (see Figure 4)
\[
\Omega_{\widetilde{\eta}(s)} = \{ x \in \mathbb{R}^2 :|x| < 1 + \widetilde{\eta}(s) \},
\]
where
\[
\widetilde{\eta}(s) = s \alpha_k \cos(k\theta) + o(s),
\]
with $k\in N^+_{even}$ denoting the set of positive even integers and \(\alpha_k\) being a nonzero constant.
\end{itemize}
\end{theorem}

\begin{figure}[ht]
\centering
\includegraphics[width=0.7\textwidth]{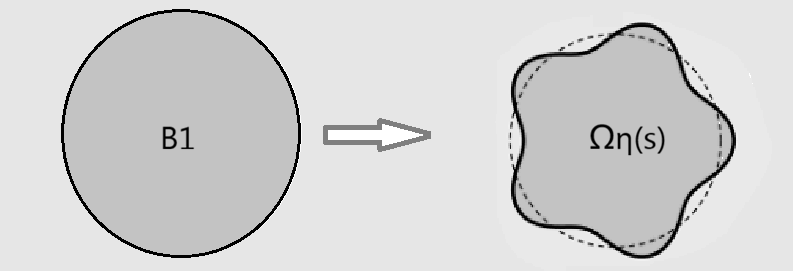}
\caption{The nontrivial domains $\Omega_{\widetilde{\eta}(s)}$  for Theorem \ref{thm1.3} bifurcating from $B_{1}$.}
\end{figure}

\begin{remark}
\noindent (i). Unlike the first two overdetermined problems considered in this paper,
problem \eqref{e1.4} is invariant under translations of the ambient space.
Consequently, the first Fourier mode
$\cos\theta$ corresponds to the infinitesimal action of
translations and therefore belongs to the kernel of the linearized operator.
The bifurcation analysis in Section~5 is performed in a $G-$invariant
subspace.
\medskip

\noindent (ii). Our approach differs notably from those in \cite{AgostinianiBM, KamburovS, Meyer}. Their proof rely on the Hanzawa transformation, which leads to a cumbersome linearization. In contrast, we employ the shape derivative method, which substantially simplifies the linearized process.
\end{remark}

At last, we give a rigid result of \eqref{e1.4} as follows.

\begin{theorem} \label{thm1.4}
 Let $\Omega\subset\mathbb R^2$ be a bounded simply connected domain whose boundary $\partial\Omega$ is connected and of class $C^{3,\alpha}$. Suppose that $(\widetilde{\psi}_{1}, \widetilde{\psi}_{2})$ solve \eqref{e1.4}
and there exist constants $m_1,m_2\in\mathbb R$ such that
 \[ \partial_{\nu\nu}\widetilde{\psi}_{1}=m_1, \qquad \partial_{\nu\nu}\widetilde{\psi}_{2}=m_2 \qquad\text{on }\partial\Omega. \] Then the curvature $\mathcal{K}$ is constant on $\partial\Omega$. Consequently, $\Omega$ is a disk. Moreover, $\widetilde{\psi}_{1}$ is radially symmetric with respect to the center of $\Omega$. If the logarithmic coefficient of $\widetilde{\psi}_{2}$ at infinity is fixed, then $\widetilde{\psi}_{2}$ is also radially symmetric with respect to the same center.
 \end{theorem}

Beyond its physical relevance to closed-streamline flows, our method naturally generalizes to higher-dimensional overdetermined elliptic problems of a similar nature. Moreover, the results are expected to carry over to more general vorticity distributions \(\gamma(\psi)\), particularly the affine case \(\gamma(\psi)=\gamma\psi\) for constant \(\gamma\), as considered in \cite{Eh,EncisoARS}. This direction offers a promising avenue for future research.

The remainder of the paper is organized as follows. Section 2 collects the necessary preliminaries, including the function spaces, shape derivatives, and spherical harmonics. Section 3 contains the proofs of Theorem \ref{thm1.1} and Corollary \ref{co}, based on shape derivatives, the Crandall--Rabinowitz bifurcation theorem, and the implicit function theorem. Section 4 presents the proof of Theorem \ref{thm1.2}; since the argument closely parallels that of Theorem \ref{thm1.1}, we highlight only the essential modifications while establishing the rigidity result. The proof of Corollary \ref{co1} is omitted, as it follows the same lines as that of Corollary \ref{co}. In Section 5, we prove Theorems \ref{thm1.3} and \ref{thm1.4}, which require additional analysis due to the surface tension in problem \eqref{e1.4}. Finally, Section 6 (Appendix) gathers several auxiliary results: the Crandall--Rabinowitz local bifurcation theorem, Reichel's theorem on the rigidity of elliptic free-boundary problems \eqref{e6.1}--\eqref{e6.2}, Aleksandrov's theorem, and the implicit function theorem.

\medskip
\section{Preliminaries}

\subsection{Function spaces and abstract operators}
In this subsection, we introduce the function spaces that will be used throughout the paper. Fix the annular domain \(\Omega_0 = B_1\) and \(D_0 = B_\lambda\) with \(\lambda \in (0,1)\). Define the even H\"older spaces
\[
C_{\mathrm{even}}^{k,\alpha}(\partial \Omega_0) := \bigl\{ u \in C^{k,\alpha}(\partial \Omega_0) : u(r,\theta) = u(r,-\theta) \bigr\}.
\]
As in \cite[Lemma 2.2]{KamburovS}, we further introduce a group $G:=\mathbb Z_2\times \mathbb Z_2$ act on $S^1$ through the reflections
\[
\rho_1(\theta)= -\theta,
\qquad
\rho_2(\theta)= \pi-\theta .
\]
and define the G-invariant subspaces
\[
C_{\mathrm{G}}^{k,\alpha}(\partial \Omega_0) := \bigl\{ u \in C^{k,\alpha}(\partial \Omega_0) : u(r,\theta)=u(r,-\theta)=u(r,\pi-\theta)\}.
\]
We then set
\begin{equation}
\label{e2.1}
\begin{aligned}
X := \left\{\left. \eta \in C_{\mathrm{even}}^{2,\alpha}(\partial \Omega_0) \right|\int_{\partial \Omega_0} \eta \, ds = 0 \right\}, \\
Y :=\left\{\left.  \eta \in C_{\mathrm{even}}^{1,\alpha}(\partial \Omega_0)\right| \int_{\partial \Omega_0} \eta \, ds = 0 \right\}
\end{aligned}
\end{equation}
and
\begin{equation}
\label{e2.2}
\begin{aligned}
\widetilde{X} := \left\{\left. \eta \in C_{\mathrm{G}}^{2,\alpha}(\partial \Omega_0) \right|\int_{\partial \Omega_0} \eta \, ds = 0 \right\}, \\
\widetilde{Y} :=\left\{\left.  \eta \in C_{\mathrm{G}}^{0,\alpha}(\partial \Omega_0)\right| \int_{\partial \Omega_0} \eta \, ds = 0 \right\}.
\end{aligned}
\end{equation}

\subsection{Shape derivatives}

We recall the notion of shape derivatives \cite{CavallinaY,HenrotP}. Let
\[
\mathcal{A} := \left\{\left. \Phi \in C^{2,\alpha}\bigl([0,1), C^{2,\alpha}(\mathbb{R}^N,\mathbb{R}^N)\bigr) \right| \Phi(0,\cdot) \equiv 0 \right\}.
\]
For \(\Phi \in \mathcal{A}\), \(t \in [0,1)\), and a domain \(\Omega \subset \mathbb{R}^N\), define \(\Phi(t) := \Phi(t,\cdot)\) and the perturbed domain
\[
\Omega_t := (\mathrm{Id} + \Phi(t))\Omega = \{ x + \Phi(t,x) : x \in \Omega \}.
\]
By definition of \(\mathcal{A}\), there exists a vector field \(h \in C^2(\mathbb{R}^N,\mathbb{R}^N)\) such that
$$
\Phi(t) = t h + o(t) \quad \text{as} \quad t \to 0.
$$

For a shape functional \(J\) and a deformation \(\Phi \in \mathcal{A}\), the \emph{shape derivative} of \(J\) at \(\Omega\) in the direction \(\Phi\) is defined by
\[
J'(\Phi) := \frac{d}{dt} J((\mathrm{Id} + \Phi(t))\Omega) \Big|_{t=0} = \lim_{t \to 0} \frac{J((\mathrm{Id} + \Phi(t))\Omega) - J(\Omega)}{t}.
\]

Let \(v(t,x) \in C^{1,\alpha}([0,1), C^{1,\alpha}(\Omega_t,\mathbb{R}))\) be a state function defined on the perturbed domains. The \emph{shape derivative} of \(v\) at time \(t_0 \in [0,1)\) and point \(x \in \Omega\) is
\[
v'(t_0,x) := \frac{\partial v}{\partial t}(t_0,x),
\]
while the \emph{material derivative} is given by
\[
\dot{v}(t_0,x) := \frac{\partial w}{\partial t}(t_0,x),
\]
where \(w(t,x) := v(t, x + \Phi(t,x))\). These two notions are related by the fundamental identity
$$
v' = \dot{v} - \nabla v \cdot h.
$$

In this paper we shall use the shape derivative of the solutions $u_t$ to the boundary value problems posed on $\Omega_t = (\mathrm{Id} + t h)\Omega$. Since the domains $\Omega_t$ vary with $t$, the shape derivative $u'$ is defined indirectly via
\begin{equation}\label{e2.3}
u' = \dot{u} - \nabla u \cdot h,
\end{equation}
which plays crucially throughout the analysis.

\subsection{Spherical harmonics}

We recall the spherical harmonics following \cite{AT}. The real spherical harmonics of degree \(k \geq 0\) are the eigenfunctions of the negative Laplace--Beltrami operator \(-\Delta_{\mathbb{S}^{N-1}}\) on the unit sphere \(\mathbb{S}^{N-1}\) associated with the eigenvalue
\[
\lambda_k = k(N + k - 2).
\]
that is,
\[
-\Delta_{\mathbb{S}^{N-1}} Y_k^i = \lambda_k Y_k^i.
\]
It is well-known (see \cite[Chapter 4]{ambook}) that the dimension of the eigenfunction for the $k$-th eigenvalue is given by
$$i_k=N_k-N_{k-2},\quad N_k=\begin{cases}
\frac{(N+k-1)!}{(N-1)!k!},\quad &\mbox{if}~k\geq0,\\
0,&\mbox{if}~k<0.
\end{cases}$$
The family \(\{Y_k^i : k \geq 0,\, 1\leq i \leq i_k\}\) forms a complete orthonormal basis of \(L^2(\mathbb{S}^{N-1})\) when normalized so that \(\|Y_k^i\|_{L^2(\mathbb{S}^{N-1})} = 1\). Consequently, every \(u \in L^2(\mathbb{S}^{N-1})\) admits the expansion
\[
u = \sum_{k \geq 0} \sum_{i=1}^{i_k} \alpha_k^i Y_k^i, \qquad \alpha_k^i = \int_{\mathbb{S}^{N-1}} u \, Y_k^i \, d\mathbb{S}^{N-1}.
\]
For each \(k \geq 0\), we denote by \(\mathbf{Y}_k := \operatorname{span}\{Y_k^i : 1\leq i \leq i_k\}\) the eigenspace corresponding to \(\lambda_k\). In particular, \(\mathbf{Y}_0\) is the one-dimensional space of constant functions on \(\mathbb{S}^{N-1}\). Note that any function with zero mean on \(\mathbb{S}^{N-1}\) has vanishing projection onto \(\mathbf{Y}_0\).

In the two-dimensional case \(N=2\), for any \(u \in X\) (see \eqref{e2.1}), it is known that the even function \(u\) admits the Fourier cosine expansions
\begin{equation}\label{e2.4}
u(1,\theta) = \sum_{k \geq 1} \alpha_k \cos(k\theta), \qquad \alpha_k = \int_{\mathbb{S}^1} u(1,\theta) \cos(k\theta) \, d\theta,
\end{equation}
where \(\theta = x/|x|\) and the coefficients \(\alpha_k \in \mathbb{R}\) are real constants.
For any \(u \in \widetilde{X}\) (see \eqref{e2.2}), it is known that the function \(u\) admits the Fourier expansions
\begin{equation}\label{e2.5}
u(1,\theta) = \sum_{k \geq 2} \alpha_k \cos(k\theta), \qquad \alpha_k = \int_{\mathbb{S}^1} u(1,\theta) \cos(k\theta) \, d\theta,
\end{equation}
where $k\in N^+_{even}$ denoting the set of positive even integers. We shall use these expansions to investigate the spectral theory of the corresponding linear operator in the following sections.

\medskip
\section{The partially overdetermined elliptic problem (\ref{e1.2})}

In this section we construct nontrivial solutions \(\psi\) and domains \(\Omega \setminus B_\lambda\) to the partially overdetermined problem \eqref{e1.2} by local bifurcation from the branch of the following radial solutions defined on the annulus \(\Omega_0 \setminus B_\lambda\).

\subsection{The trivial branch and abstract operator}\label{sub3.1}

The explicit form of the trivial solutions is given as follows.

\begin{lemma}\label{lem3.1}
For any \(\lambda \in (0,1)\) and \(\gamma \in \mathbb{R}\), there exists a unique constant \(Q^\gamma\) given by
\[
Q^\gamma = \left( \frac{4 + (1 - \lambda^2)\gamma}{4\ln\lambda} + \frac{\gamma}{2} \right)^2
\]
such that the problem \eqref{e1.2} on the annulus \(B_1 \setminus B_\lambda\) admits a unique radially symmetric solution \(\psi = \psi^\gamma_\mathrm{tri}\) satisfying
\[
\psi = 1 \quad \text{on } \partial B_\lambda, \qquad
\psi = 0 \quad \text{on } \partial B_1, \qquad
|\nabla \psi|^2 = Q^\gamma \quad \text{on } \partial B_1.
\]
The solution is explicitly given by
\[
\psi^\gamma_\mathrm{tri}(|x|) = \frac{4 + (1 - \lambda^2)\gamma}{4\ln\lambda} \ln |x| - \frac{\gamma(1 - |x|^2)}{4}.
\]
\end{lemma}

\begin{proof}
Assume \(\psi(r)\) with \(r = |x|\) is a radially symmetric solution. Then \(\psi\) satisfies the ODE
\[
\partial_{rr}\psi + \frac{1}{r}\partial_{r}\psi = \gamma.
\]
Integrating once yields
\[
\partial_{r}\psi= \frac{C}{r} + \frac{\gamma}{2} r.
\]
A further integration and application of the Dirichlet boundary conditions give
\[
\psi(r) = C \ln r - \frac{\gamma(1 - r^2)}{4},
\]
where
\[
C = \frac{4 + (1 - \lambda^2)\gamma}{4\ln\lambda}.
\]
Based on this explicit form, after direct computation we get
\[
Q^\gamma = \left( \frac{4 + (1 - \lambda^2)\gamma}{4\ln\lambda} + \frac{\gamma}{2} \right)^2.
\]
This completes the proof.
\end{proof}

Next, we reformulate problem \eqref{e1.2} as an abstract operator equation on the spaces \(X\) and \(Y\) introduced in \eqref{e2.1}.
It is known that, for any \(\eta \in X\), standard elliptic theory (see, e.g., \cite{Gilbarg}) yields a unique even solution \(\psi_\eta \in C_{\mathrm{even}}^{2,\alpha}(\Omega_\eta \setminus \overline{B_\lambda})\) to the Dirichlet problem
\begin{equation}\label{e3.1}
\begin{cases}
\Delta \psi = \gamma & \text{in } \Omega_\eta \setminus \overline{B_\lambda}, \\
\psi = 0 & \text{on } \partial \Omega_\eta, \\
\psi = 1 & \text{on } \partial B_\lambda.
\end{cases}
\end{equation}
Based on the \(\eta\)-dependent solutions \(\psi_\eta\), we define the nonlinear operators
\begin{equation}\label{e3.2}
f \colon \mathbb{R} \times X \to C_{\mathrm{even}}^{1,\alpha}(\partial \Omega_0), \qquad (\gamma, \eta) \mapsto \partial_{\nu_\eta} \psi_\eta \big|_{\partial \Omega_0},
\end{equation}
where \(\nu_\eta\) is the outward unit normal to \(\partial \Omega_\eta\). Since \(\partial \Omega_\eta\) is a level set of \(\psi_\eta\), we have \(\nu_\eta = \ell \frac{\nabla \psi_\eta}{|\nabla \psi_\eta|}\) on \(\partial \Omega_\eta\) for some sign \(\ell \in \{+,-\}\). Consequently,
\begin{equation}\label{e3.3}
f(\gamma,\eta) = \ell \, |\nabla \psi_\eta(x + \eta(x)\nu(x))| \quad \text{for } x \in \partial \Omega_0.
\end{equation}
 Define the following zero-mean subspaces of H\"{o}lder space
\begin{equation}\label{zy}
 M_1=\left\{\left. \eta \in C^{1,\alpha}(\partial \Omega_0) \right|\int_{\partial \Omega_0} \eta \, ds = 0 \right\}, \quad  M_0=\left\{\left. \eta \in C^{0,\alpha}(\partial \Omega_0) \right|\int_{\partial \Omega_0} \eta \, ds = 0 \right\}
\end{equation}
 and the operator
$$
G \colon \mathbb{R} \times X \to Y, \qquad (\gamma, \eta) \mapsto \Pi_1\left( f^2(\gamma, \eta) - Q^\gamma\right),
$$
where $\Pi_1:C^{1,\alpha}(\partial \Omega_0)\rightarrow M_1$ is the projection operator, defined by
$$
 \Pi_1(\phi)=\phi-\frac{1}{|\partial \Omega_0|}\int_{\partial \Omega_0}\phi\, ds
$$
and \(f\) is the operator defined in \eqref{e3.2} and \eqref{e3.3}. Then \(G(\gamma, \eta) = 0\) if and only if \(\partial_{\nu_\eta} \psi_\eta\) is constant on \(\partial \Omega_\eta\), which is equivalent to a solution of \eqref{e1.2}.

\subsection{The linearized operator and its spectrum}

The goal of this subsection is to find nontrivial \(\eta \in X\) solving
\begin{equation}\label{e3.4}
G(\gamma, \eta) = \Pi_1\left( f^2(\gamma, \eta) - Q^\gamma\right) = 0.
\end{equation}
The Fr\'{e}chet differentiability of \(G\) near \((\gamma, 0)\) follows from the differentiability of the solution map \(\eta \mapsto \psi_\eta\) and its derivatives, using \cite[Theorem 5.3.2]{HenrotP} and elliptic regularity \cite{Gilbarg}. Note that \(G(\gamma, 0) = 0\) by Lemma \ref{lem3.1}.

\begin{lemma}\label{lem3.2}
Let \(\eta_0 \in X\) be fixed and \(\psi_t = \psi_{t\eta_0}\) solve \eqref{e3.1} for \(t \in (-\varepsilon, \varepsilon)\). Then the shape derivative \(\psi'\) of \(\psi_t\) at \(t=0\) exists in \(\Omega_0 \setminus \overline{B_\lambda}\) and satisfies
\begin{equation}\label{e3.5}
\begin{cases}
\Delta \psi' = 0 & \text{in } \Omega_0 \setminus \overline{B_\lambda}, \\
\psi' = -\partial_\nu \psi^\gamma_\mathrm{tri} \, \eta_0 = -\sqrt{Q^\gamma} \, \eta_0 & \text{on } \partial \Omega_0, \\
\psi' = 0 & \text{on } \partial B_\lambda.
\end{cases}
\end{equation}
\end{lemma}

\begin{proof}
Differentiating the equation \(\Delta \psi_t = \gamma\) in \(\Omega_t \setminus \overline{B_\lambda}\) with respect to \(t\) (at \(t=0\)) immediately yields \(\Delta \psi' = 0\) in \(\Omega_0 \setminus \overline{B_\lambda}\).

The boundary conditions follow by differentiating the Dirichlet data in \eqref{e3.1} along the perturbed domains using the shape derivative formula \eqref{e2.3}. On \(\partial \Omega_t\), we have \(\psi_t \equiv 0\), so its material derivative vanishes. Thus
\[
\psi' = \dot{\psi} - \nabla \psi \cdot h = -\nabla \psi^\gamma_\mathrm{tri} \cdot \eta_0 = -\partial_\nu \psi^\gamma_\mathrm{tri} \, \eta_0.
\]
Since \(|\nabla \psi^\gamma_\mathrm{tri}| = \sqrt{Q^\gamma}\) on \(\partial B_1\) and \(\nu\) is the outward normal, this reduces to \(\psi' = -\sqrt{Q^\gamma} \, \eta_0\) on \(\partial \Omega_0\).

On the fixed inner boundary \(\partial B_\lambda\), \(\psi_t \equiv 1\), so both the material and shape derivatives vanish, giving \(\psi' = 0\).
\end{proof}

\begin{proposition}\label{pro-3.3}
The map \(G\) is Fr\'{e}chet differentiable in a neighborhood of \((\gamma, 0)\). Moreover,
\begin{equation}\label{e3.6}
\partial_\eta G(\gamma, 0)[\eta_0] = 2 \partial_\nu \psi^\gamma_\mathrm{tri} \bigl( \partial_\nu \psi' + \partial_{\nu\nu} \psi^\gamma_\mathrm{tri} \, \eta_0 \bigr), \quad \eta_0 \in X,
\end{equation}
where \(\psi'\) solves \eqref{e3.5}.
\end{proposition}

\begin{proof}
By definition of \(G\) and the chain rule,
$$
\partial_\eta G(\gamma,0)[\eta_0] = 2 f(\gamma,0) \cdot \partial_\eta f(\gamma,0)[\eta_0] = 2 \partial_\nu \psi^\gamma_\mathrm{tri} \cdot \partial_\eta f(\gamma,0)[\eta_0],~\eta_0\in X,
$$
where we use \(f(\gamma,0) = \partial_\nu \psi^\gamma_\mathrm{tri} = \sqrt{Q^\gamma}\) (up to the consistent choice of sign \(\ell\)).

It remains to compute the Gâteaux derivative of \(f\):
$$
\partial_\eta f(\gamma,0)[\eta_0] = \lim_{t \to 0^+} \frac{f(\gamma, t\eta_0) - f(\gamma,0)}{t} = \frac{d}{dt}\Big|_{t=0} f(\gamma, t\eta_0).
$$
From \eqref{e3.3},
$$
f(\gamma, t\eta_0) = \ell \, \bigl|\nabla \psi_t(x + t\eta_0(x)\nu(x))\bigr|.
$$
Differentiating under the norm and evaluating at \(t=0\) gives
\[
\partial_\eta f(\gamma,0)[\eta_0] = \nu(x) \cdot \dot{b}(x),
\]
where
\[
\dot{b}(x) = \frac{d}{dt}\Big|_{t=0} \nabla \psi_t(x + t\eta_0(x)\nu(x)) = \nabla \psi'(x) + \eta_0(x) D^2 \psi^\gamma_\mathrm{tri}(x) \nu(x).
\]
Thus
$$
\partial_\eta f(\gamma,0)[\eta_0] = \partial_\nu \psi' + \partial_{\nu\nu} \psi^\gamma_\mathrm{tri} \, \eta_0.
$$
Substituting back into the expression for \(\partial_\eta G\) yields \eqref{e3.6}.
\end{proof}

To compute the spectrum of \(\partial_\eta G(\gamma, 0)\), we expand the shape derivative \(\psi'\) in spherical harmonics.

\begin{proposition}\label{pro-3.4}
Let \(\eta_0 \in X\) with Fourier expansion
\[
\eta_0(\theta) = \sum_{k \geq 1} \alpha_k \cos(k\theta).
\]
Then the solution \(\psi'\) of \eqref{e3.5} admits the expansion
\[
\psi'(r, \theta) = \sum_{k \geq 1}\left (A_k r^{-k} + B_k r^k\right) \alpha_k \cos(k\theta),
\]
where
\[
A_k = \frac{2\gamma \lambda^{2k} \ln\lambda + 4\lambda^{2k} + \gamma(1 - \lambda^2)\lambda^{2k}}{4\ln\lambda (1 - \lambda^{2k})}, \qquad
B_k = \frac{-2\gamma \ln\lambda - 4 - \gamma(1 - \lambda^2)}{4\ln\lambda (1 - \lambda^{2k})}.
\]
\end{proposition}

\begin{proof}
We seek a solution of the form
\[
\psi'(r,\theta) = \sum_{k \geq 1} S_k(r) \alpha_k \cos(k\theta),
\]
where \(r = |x|\) and \(\theta = x/|x|\). Substituting into \(\Delta \psi' = 0\) and using the angular dependence of \(\cos(k\theta)\) yields the Euler equation for each radial function \(S_k(r)\):
$$
\partial_{rr}S_k(r) + \frac{1}{r} \partial_rS_k(r) - \frac{k^2}{r^2} S_k(r) = 0, \qquad r \in (\lambda, 1).
$$
The general solution is
\[
S_k(r) = A_k r^{-k} + B_k r^k.
\]
The boundary conditions in \eqref{e3.5} determine the coefficients. At \(r=1\),
\[
S_k(1) = A_k + B_k = -\sqrt{Q^\gamma} = -\left( \frac{4 + (1 - \lambda^2)\gamma}{4\ln\lambda} + \frac{\gamma}{2} \right),
\]
while at \(r=\lambda\),
\[
S_k(\lambda) = A_k \lambda^{-k} + B_k \lambda^k = 0.
\]
Solving this linear system gives the explicit expressions for \(A_k\) and \(B_k\) stated above.
\end{proof}

Combining Propositions \ref{pro-3.3} and \ref{pro-3.4} yields the dispersion relation
\begin{equation}\label{e3.7}
\partial_\eta G(\gamma, 0)[\eta_0] = \left( \frac{4 + (1 - \lambda^2)\gamma}{2\ln\lambda} + \gamma \right) \sum_{k \geq 1} \sigma_k \alpha_k \cos(k\theta),
\end{equation}
with
\begin{equation}\label{e3.8}
\sigma_k = (-k A_k + k B_k) - \left( \frac{4 + (1 - \lambda^2)\gamma}{4\ln\lambda} - \frac{\gamma}{2} \right).
\end{equation}
It follows from $Q>0$ that $\frac{4+(1-\lambda^2)\gamma}{2\ln\lambda}+\gamma\neq0$, then the kernel of \(\partial_\eta G(\gamma, 0)\) is nontrivial precisely when \(\sigma_k(\gamma) = 0\) for some \(k \geq 1\).

A direct computation shows that \(\sigma_k(\gamma) = 0\) if and only if \(g_k(\gamma) = 0\), where
\[
g_k(\gamma) = k \bigl(4 + \gamma(1 - \lambda^2) + 2\gamma \ln\lambda\bigr)(1 + \lambda^{2k}) + \bigl(4 + \gamma(1 - \lambda^2) - 2\gamma \ln\lambda\bigr)(1 - \lambda^{2k}).
\]
The roots are
\begin{equation}\label{e3.9}
\gamma_k^* = -\frac{4(k + k\lambda^{2k} + 1 - \lambda^{2k})}{\bigl(1 - \lambda^2 + 2\ln\lambda\bigr)(k + k\lambda^{2k}) + \bigl(1 - \lambda^2 - 2\ln\lambda\bigr)(1 - \lambda^{2k})}.
\end{equation}

We now prove that the bifurcation point \(\gamma_1^*\) given in \eqref{e3.9} is simple for all \(\lambda \in (0,1)\), while \(\gamma_k^*\) for \(k \geq 2\) is simple for \(\lambda\) in a suitable subinterval of \((0,1)\).

The numerator appearing in \eqref{e3.9} is positive for every \(\lambda \in (0,1)\) and \(k \in \mathbb{N}^+\). We denote the denominator by
\[
f_{de}(\lambda,k) = A(k + k\lambda^{2k}) + B(1 - \lambda^{2k}),
\]
where
\[
A = 1 - \lambda^2 + 2\ln\lambda < 0, \qquad B = 1 - \lambda^2 - 2\ln\lambda > 0
\]
for all \(\lambda \in (0,1)\).

For \(k=1\),
\[
f_{de}(\lambda,1) = 2(1 - \lambda^2 + 2\lambda^2 \ln\lambda) > 0,
\]
which yields
\[
\gamma_1^* = \frac{4}{\lambda^2 - 2\lambda^2 \ln\lambda - 1} < 0.
\]
For \(k \in \mathbb{N}^+ \setminus \{1\}\) and \(\lambda \in (0,1)\), substituting \(\gamma = \gamma_1^*\) into \(g_k\) gives
\[
g_k(\gamma_1^*) = \frac{8\ln\lambda \bigl( k(1 + \lambda^{2k})(1 - \lambda^2) - (1 - \lambda^{2k})(1 + \lambda^2) \bigr)}{\lambda^2 - 2\lambda^2 \ln\lambda - 1} > 0.
\]
Hence \(\gamma_1^*\) is a simple zero of \(g_1(\gamma)\). This follows from the auxiliary inequality
\begin{equation}
\label{e3.10}
k(1 + \lambda^{2k})(1 - \lambda^2) - (1 - \lambda^{2k})(1 + \lambda^2) > 0,
\end{equation}
which holds for all \(k \in \mathbb{N}^+ \setminus \{1\}\) and \(\lambda \in (0,1)\). In order to show it, we notice that
$$1+\lambda^{2k}> \lambda^{2m}+\lambda^{2k-2m}\quad\mathrm{for}\quad\lambda\in(0,1)\quad\mathrm{and}\quad m=1,\cdots,k-1.$$
As a consequence,
\begin{align*}
k(1+\lambda^{2k})&>(1+\lambda^{2k})+\sum_{m=1}^{k-1}(\lambda^{2m}+\lambda^{2k-2m})\\
&=(1+\lambda^2)(1+\lambda^{2}+\cdots+\lambda^{2k-2})=\frac{1-\lambda^{2k}}{1-\lambda^2}(1+\lambda^2),
\end{align*}
which gives the desired inequality \eqref{e3.10}.

We next establish the simplicity of \(\gamma_k^*\) for \(k \geq 2\). Observe that \(f_{de}(\lambda,k)\) changes sign in \((0,1)\). To facilitate the analysis, we restrict \(\lambda\) to the values for which \(f_{de}(\lambda,k)\) maintains a constant sign. Note that
\[
f_{de}(\lambda,k) = B(k + k\lambda^{2k}) \left( \frac{A}{B} + \frac{1 - \lambda^{2k}}{k + k\lambda^{2k}} \right),
\]
where the prefactor \(B(k + k\lambda^{2k})\) is positive. Moreover, the function \(\frac{1 - \lambda^{2k}}{k + k\lambda^{2k}}\) is monotone decreasing in \(k\) for each fixed \(\lambda \in (0,1)\). Indeed, its derivative with respect to \(k\) is
\[
-\frac{1}{k^2(1+\lambda^{2k})^2} (1 - \lambda^{4k} + 4k \log\lambda \cdot \lambda^{2k}).
\]
Regarding $\lambda^{2k}$ as $t$, the term inside the bracket of the above equation can be written as $1-t^2+2t\log t $, and its derivative $$-2t+2+2\log t<0\quad\mbox{for}\quad t\in(0,1).$$
Together with the simple fact that
$$1-t^2+2t\log t=0\quad\mbox{at}\quad t=1,$$
we get
$$1-t^2+2t\log t >0.$$
On the other hand, a direct numerical check shows that \(f_{de}(0.2483,2) \approx 0\). Consequently, for all \(\lambda \in (0,0.2483)\) and \(k \in \mathbb{N}^+ \setminus \{1\}\),
\[f_{de}(\lambda,k) < 0,\]
which implies \(\gamma_k^* > 0\) in this specific range. Furthermore, for \(\lambda \in (0,0.2483)\) and \(k \in \mathbb{N}^+ \setminus \{1\}\),
$$
\frac{d \gamma_k^*}{dk} = \frac{16\ln\lambda \bigl(1 - \lambda^{4k} + 4k \log\lambda \cdot \lambda^{2k}\bigr)}{\bigl(Ak + B + (Ak - B)\lambda^{2k}\bigr)^2} < 0.
$$
It follows that each \(\gamma_k^*\) is simple for \(\lambda \in (0,0.2483)\).

In the remainder of this section we shall restrict discussion to the case \(k=1\). The corresponding bifurcation value is
\begin{equation}\label{e3.11}
\gamma^* := \frac{4}{\lambda^2 - 2\lambda^2 \ln\lambda - 1},
\end{equation}
where the denominator is negative for every \(\lambda \in (0,1)\).

\subsection{Proof of Theorem \ref{thm1.1}}

We complete the proof of Theorem \ref{thm1.1} by applying the Crandall--Rabinowitz local bifurcation theorem (Theorem \ref{thm6.1} in the Appendix).

\begin{proof}[Proof of Theorem \ref{thm1.1}]
From the preceding analysis,
$$
G(\gamma,0)=0
$$
holds for all \(\gamma \in \mathbb{R} \setminus \bigl\{ \frac{4}{\lambda^2 - 2\ln\lambda - 1} \bigr\}\), where \(G\) is the operator defined in \eqref{e3.4}. Thus condition (H1) of Theorem \ref{thm6.1} is satisfied.

It remains to verify condition (H2) at the critical value \(\gamma = \gamma^*\) given in \eqref{e3.11}. From the dispersion relation \eqref{e3.7}--\eqref{e3.8}, the kernel \(\mathcal{N}(\partial_\eta G(\gamma^*,0))\) is one-dimensional and spanned by
\[
\eta^*(\theta) = \alpha_1 \cos\theta \in X.
\]
The range \(\mathcal{R}(\partial_\eta G(\gamma^*,0))\) is the closed subspace of \(Y\) consisting of all \(\phi \in Y\) such that
\[
\int_0^{2\pi} \phi(\theta)\cos\theta \, d\theta = 0.
\]
Consequently, \(Y / \mathcal{R}(\partial_\eta G(\gamma^*,0))\) is one-dimensional and generated by \(\eta^*\).

Furthermore, a direct computation using \eqref{e3.7}--\eqref{e3.8} yields
\begin{equation}
\begin{aligned}
\partial_{\gamma\eta} G(\gamma^*,0)[1,\eta^*] &= \frac{\bigl(-2\lambda^2 \ln\lambda - 1 + \lambda^2\bigr) \bigl(4 + (1 - \lambda^2 + 2\ln\lambda)\gamma^*\bigr)}{4\ln^2\lambda (1 - \lambda^2)} \eta^* \\
&=\frac{2}{\ln\lambda}\eta^*, \nonumber
\end{aligned}
\end{equation}
which is obviously not in $\mathcal{R}(\partial_\eta G(\gamma^*,0)).$
So condition (H2) holds.

Applying the Crandall--Rabinowitz bifurcation theorem \ref{thm6.1} completes the proof of Theorem \ref{thm1.1}.
\end{proof}

\subsection{Proof of Corollary \ref{co}}

We prove Corollary \ref{co} by means of the implicit function theorem (Theorem \ref{thm6.5} in the Appendix).

\begin{proof}[Proof of Corollary \ref{co}]
By the results of Subsection \ref{sub3.1}, the perturbed problem \eqref{else} can be recast as an abstract operator equation. Define
\begin{equation}
G_p \colon Y \times X \to Y, \qquad (\rho,\eta) \mapsto \Pi_1\left(f^2(\eta) - \rho \right) , \nonumber
\end{equation}
where \(f\) is the operator introduced in \eqref{e3.2}. Clearly \(G_p(0,0) = 0\).

To apply the implicit function theorem, it suffices to show that the partial derivative \(\partial_\eta G_p(0,0)\) is a bounded linear isomorphism from \(X\) onto \(Y\). The linearity and boundedness follow immediately from the expression in \eqref{e3.6} and the boundary value problem \eqref{e3.5}. It remains to establish bijectivity.

Let \(\eta_0(\theta) = \sum_{k \geq 1} \alpha_k \cos(k\theta) \in X\). From Propositions \ref{pro-3.3}--\ref{pro-3.4} and \eqref{e3.7}-\eqref{e3.8}, we have
\begin{equation}\label{e3.12}
\partial_\eta G_p(0,0)[\eta_0] = 2\partial_\nu \psi^0_\mathrm{tri} \bigl( \partial_\nu \psi' + \partial_{\nu\nu} \psi^0_\mathrm{tri} \, \eta_0 \bigr) = 2\sqrt{Q^\gamma} \sum_{k \geq 1} \sigma_k \alpha_k \cos(k\theta),
\end{equation}
where
\begin{equation}\label{e3.13}
\sigma_k = \frac{k\bigl(\gamma(\lambda^2-1)-4-2\gamma \ln\lambda\bigr)(1+\lambda^{2k})+ \bigl(\gamma(\lambda^2-1)-4+2\gamma \ln\lambda\bigl)(1-\lambda^{2k})}{4\ln\lambda(1-\lambda^{2k})}.
\end{equation}
By assumption \(\sigma_k \neq 0\) for all \(k \geq 1\), so \(\partial_\eta G_p(0,0)\) is injective.

For surjectivity, take an arbitrary \(y_0 \in Y\) with Fourier expansion
\[
y_0(\theta) = \sum_{k \geq 1} \tau_k \cos(k\theta).
\]
Define
\begin{equation}\label{e3.14}
\eta_0 := \frac{1}{2\sqrt{Q^\gamma}} \sum_{k \geq 1} \frac{\tau_k}{\sigma_k} \cos(k\theta).
\end{equation}
Since \(\{\,1/\sigma_k\,\}\) is bounded, \(\eta_0\) belongs to \(L^2_\mathrm{even}(\partial \Omega_0)\). Moreover, the summation begins at \(k=1\), so \(\int_{\partial \Omega_0} \eta_0 \, ds= 0\) and thus \(\eta_0 \in X\) after suitable extension. Let \(\mathbb{L}\) denote the continuous extension of the operator in \eqref{e3.12} to \(L^2_\mathrm{even}(\partial \Omega_0)\). By construction, \(\eta_0 = \mathbb{L}^{-1}(y_0)\).

It remains to show that if \(y_0 \in C^{1,\alpha}_\mathrm{even}(\partial \Omega_0)\), then \(\eta_0 \in C^{2,\alpha}_\mathrm{even}(\partial \Omega_0)\). We proceed as in \cite[Proposition 5.2]{KamburovS} or \cite[Proof of Theorem II]{Cavallina}. Recall the Sobolev spaces
\[
H^s_\mathrm{even}(\partial \Omega_0) := \left\{ \sum_{k=1}^\infty \alpha_k \cos(k\theta) : \alpha_k \in \mathbb{R},\ \sum_{k=1}^\infty (1 + k^2)^s \alpha_k^2 < \infty \right\}.
\]
Since \(y_0 \in C^{1,\alpha}_\mathrm{even}(\partial \Omega_0) \subset H^1_\mathrm{even}(\partial \Omega_0)\) and \(\sigma_k\) is given by \eqref{e3.13}, it follows that \(\eta_0 \in H^2_\mathrm{even}(\partial \Omega_0)\). In addition, from \eqref{e3.5}, \eqref{e3.6} and the definition \eqref{e3.14}, the associated shape derivative \(\psi'\) satisfies
\begin{equation}
\begin{cases}
\Delta \psi' = 0 & \text{in } \Omega_0 \setminus \overline{B_\lambda}, \\
\psi' = 0 & \text{on } \partial B_\lambda, \\
\partial_\nu \psi' = \frac{y_0}{2\partial_\nu \psi^\gamma_\mathrm{tri}} - \partial_{\nu\nu} \psi^\gamma_\mathrm{tri} \, \eta_0 & \text{on } \partial \Omega_0,  \nonumber
\end{cases}
\end{equation}
with trace \(\psi'|_{\partial \Omega_0} = -\partial_\nu \psi^\gamma_\mathrm{tri} \, \eta_0\). By the Schauder estimates and \(L^p\) theory for the Laplace equation together with Sobolev embeddings, we obtain \(\psi' \in C^{2,\alpha}_\mathrm{even}(\overline{\Omega_0 \setminus B_\lambda})\), which implies \(\eta_0 \in C^{2,\alpha}_\mathrm{even}(\partial \Omega_0)\). Thus \(\partial_\eta G_p(0,0)\) is surjective, hence an isomorphism.

Applying the implicit function theorem (Theorem \ref{thm6.5}) yields a unique \(C^1\) curve \(\eta(\rho)\) defined for \(\|\rho\|_Y < \varepsilon\) such that
\begin{equation}\label{e3.15}
G_p(\rho, \eta(\rho)) = 0.
\end{equation}
Since
\[ G_p(\rho,\eta(\rho)) = \Pi_1\left(f^2(\eta(\rho))-\rho\right) =0, \]
the function \(f^2(\eta(\rho))-\rho\) is constant on \(\partial\Omega_0\).
We define
\[ Q_\rho := \frac{1}{|\partial\Omega_0|} \int_{\partial\Omega_0} \left(f^2(\eta(\rho))-\rho\right)\,ds. \]
Since \(\rho\in Y\) has zero mean, this can equivalently be written as
\[ Q_\rho = \frac{1}{|\partial\Omega_0|} \int_{\partial\Omega_0} f^2(\eta(\rho))\,ds. \]
Consequently,
\[ f^2(\eta(\rho)) = Q_\rho+\rho \qquad\text{on }\partial\Omega_0. \]
After pushing this identity forward to \(\partial\Omega_{\eta(\rho)}\),
we obtain
\[ |\nabla\psi|^2 = Q_\rho+\rho \qquad\text{on }\partial\Omega_{\eta(\rho)}. \]
Moreover, \(Q_0=Q^\gamma\).
The local uniqueness of \(\eta(\rho)\) follows from the implicit function theorem. Once \(\eta(\rho)\) is fixed, the corresponding Bernoulli constant is uniquely determined by
\[ Q_\rho = \frac{1}{|\partial\Omega_0|} \int_{\partial\Omega_0} \left(f^2(\eta(\rho))-\rho\right)\,ds. \]
Hence the pair \((Q_\rho,\eta(\rho))\) is locally unique.

Since \(\rho(\theta) = \sum_{k \geq 1} \tau_k \cos(k\theta)\), differentiating \eqref{e3.15} with respect to \(\rho\) gives
\[
\partial_\rho \eta = -(\partial_\eta G_p)^{-1} \partial_\rho G_p = (\partial_\eta G_p)^{-1}.
\]
Consequently,
\[
\eta(\rho) = \sum_{k \geq 1} \frac{1}{2\sqrt{Q^{\gamma}}\sigma_k} \tau_k \cos(k\theta) + o(\|\rho\|_Y) \quad \text{as} \quad \|\rho\|_Y \to 0.
\]
This completes the proof.
\end{proof}

\medskip
\section{The two-phase overdetermined elliptic problem (\ref{e1.3})}

Define
\begin{eqnarray}\label{e4.1}
\Psi(x)=\left\{\begin{array}{ll}{\psi_1(x)},  & {\text { in } D}, \\
{\psi_2(x)}, & {\text { in } \Omega_{\eta} \setminus\overline{D}},\end{array}\right.  \quad
\Gamma=\left\{\begin{array}{ll}{\gamma_1},  & {\text { in } D}, \\
{\gamma_2}, & {\text { in } \Omega_{\eta} \setminus\overline{D}}.\end{array}\right.
\end{eqnarray}
In this section, we construct nontrivial solutions \(\Psi\) of the form \eqref{e4.1} and domains \(\Omega\) solving the two-phase overdetermined problem \eqref{e1.3} by local bifurcation from the branch of the following radial solutions defined on the pair of balls \((B_\lambda, B_1)\) for $\lambda\in (0,1)$.
Let's first describe the trivial branch explicitly.

\begin{lemma}
For any \(\lambda \in (0,1)\) and \(\gamma_2 \in \mathbb{R}\), there exist a unique constant
\[
Q^{\gamma_2} = \frac{\gamma_2^2}{4}
\]
and a unique value \(\Psi_\mathrm{tri}(0) = -\frac{(1-\lambda^2)\gamma_2 + \lambda^2 \gamma_1}{4}\) such that on \((\Omega, D) = (B_1, B_\lambda)\) the problem \eqref{e1.3} admits a unique radially symmetric solution \(\Psi = \Psi_\mathrm{tri}(|x|)\) given by
\begin{equation}\label{e4.2}
\Psi_\mathrm{tri}(|x|) :=
\begin{cases}
\psi_1(|x|) = -\frac{(1-\lambda^2)\gamma_2 + (\lambda^2 - |x|^2)\gamma_1}{4}, & |x| \in [0,\lambda), \\
\psi_2(|x|) = -\frac{(1 - |x|^2)\gamma_2}{4}, & |x| \in [\lambda,1).
\end{cases}
\end{equation}
\end{lemma}

As $D=B_{\lambda}$, for any $\eta\in X$, it follows from \cite{AS} that there exists a unique solution $\Psi_{\eta}\in \left(C_{\mathrm{even}}^{2,\alpha}(\Omega_{\eta}\setminus\overline{B_{\lambda}})\times C_{\mathrm{even}}^{2,\alpha}(B_{\lambda})\right)\cap C_{\mathrm{even}}^{0,\alpha}(\Omega_{\eta})$ to the following two-phase transmission problem
\begin{equation}\label{e4.3}
\begin{cases}
\Delta \Psi = \gamma_1 & \text{in } B_\lambda, \\
\Delta \Psi = \gamma_2 & \text{in } \Omega_\eta \setminus \overline{B_\lambda}, \\
\left[ \Psi \right] =\left [\tfrac{1}{\Gamma} \partial_\nu \Psi \right] = 0 & \text{on } \partial B_\lambda, \\
\Psi = 0 & \text{on } \partial \Omega_\eta,
\end{cases}
\end{equation}
where \(\left[\cdot \right]\) denotes the jump across the interface \(\partial B_\lambda\) (see \cite{AS}) and $\Psi$ is piecewise $C^{2,\alpha}$. Having constructed the \(\eta\)-dependent solution \(\Psi_\eta\), we introduce the second operator
\begin{equation}\label{e4.4}
F \colon \mathbb{R} \times X \to C_{\mathrm{even}}^{1,\alpha}(\partial \Omega_0), \qquad (\gamma_2, \eta) \mapsto \partial_{\nu_\eta} \Psi_\eta \big|_{\partial \Omega_0},
\end{equation}
which satisfies the analogous identity
\begin{equation}\label{e4.5}
F(\gamma_2,\eta) = \ell \, |\nabla \Psi_\eta(x + \eta(x)\nu(x))| \quad \text{for } x \in \partial \Omega_0.
\end{equation}
Similarly, we reformulate problem \eqref{e1.3} as an abstract operator equation on the spaces \(X\) and \(Y\). Define
\begin{equation}\label{e4.6}
H \colon \mathbb{R} \times X \to Y, \qquad (\gamma_2, \eta) \mapsto \Pi_1\left( F^2(\gamma_2, \eta) - Q^{\gamma_2}\right),
\end{equation}
where \(F\) is the operator introduced in \eqref{e4.4}. Then \(H(\gamma_2, \eta) = 0\) if and only if \(\partial_{\nu_\eta} \Psi_\eta\) is constant on \(\partial \Omega_\eta\), which is equivalent to \(\Psi_\eta\) solving \eqref{e1.3}.

\subsection{The spectrum of the linearized operator \(\partial_\eta H(\gamma_2,0)\)}

The goal of this subsection is to find nontrivial \(\eta \in X\) solving
\begin{equation}
H(\gamma_2, \eta) = \Pi_1\left( F^2(\gamma_2, \eta) - Q^{\gamma_2}\right)= 0. \nonumber
\end{equation}
The map \(F\) is Fr\'{e}chet differentiable near \((\gamma_2, 0)\) by the same arguments as in \cite[Theorem 5.3.2]{HenrotP}, combined with Schauder theory for elliptic equations with piecewise constant coefficients \cite{Lady}. Consequently, the partial Fr\'{e}chet derivatives coincide with the Gâteaux derivatives:
\[
\partial_\eta F(\gamma_2,0)[\eta_0] = \frac{d}{dt}\Big|_{t=0} F(\gamma_2, t\eta_0), \qquad \eta_0 \in X.
\]

Let \(\Psi_t = \Psi_{t\eta_0}\) solve the transmission problem \eqref{e4.3}. The shape derivative \(\Psi'\) of \(\Psi_t\) satisfies the following boundary value problem.

\begin{lemma}
Let \(\eta_0 \in X\) be fixed. Then the shape derivative \(\Psi'\) of \(\Psi_t\) for $t\in(-\varepsilon,\varepsilon)$ exists and solves
\begin{equation}\label{e4.7}
\begin{cases}
\Delta \Psi' = 0 & \text{in } B_\lambda \cup (\Omega_0 \setminus \overline{B_\lambda}), \\
\left[ \Psi' \right] = 0 & \text{on } \partial B_\lambda, \\
\left[ \tfrac{1}{\Gamma} \partial_\nu \Psi' \right] = 0 & \text{on } \partial B_\lambda, \\
\Psi' = -\partial_\nu \Psi_\mathrm{tri} \, \eta_0 = -\sqrt{Q^{\gamma_2}} \, \eta_0 & \text{on } \partial \Omega_0.
\end{cases}
\end{equation}
\end{lemma}

\begin{proof}
The result follows from the shape derivative formula \eqref{e2.3} applied to the transmission problem \eqref{e4.3}, in the same manner as in the proof of Lemma \ref{lem3.2}.
\end{proof}

\begin{proposition}\label{pro-4.3}
The map \(H\) $($defined in \eqref{e4.6}$)$ is Fr\'{e}chet differentiable in a neighborhood of \((\gamma_2, 0)\). Moreover,
\begin{equation}\label{e4.8}
\partial_\eta H(\gamma_2,0)[\eta_0] = 2 \partial_\nu \Psi_\mathrm{tri} \bigl( \partial_\nu \Psi' + \partial_{\nu\nu} \Psi_\mathrm{tri} \, \eta_0 \bigr), \quad \eta_0 \in X,
\end{equation}
where \(\Psi'\) solves \eqref{e4.7} and \(\Psi_\mathrm{tri}\) is given by \eqref{e4.2}.
\end{proposition}

\begin{proof}
By the chain rule and the definition of \(H\),
\[
\partial_\eta H(\gamma_2,0)[\eta_0] = 2 F(\gamma_2,0) \cdot \partial_\eta F(\gamma_2,0)[\eta_0] = 2 \partial_\nu \Psi_\mathrm{tri} \cdot \partial_\eta F(\gamma_2,0)[\eta_0].
\]
Since \(F(\gamma_2,0) = \partial_\nu \Psi_\mathrm{tri} = \sqrt{Q^{\gamma_2}}\) (up to sign), it remains to compute the G\^{a}teaux derivative of \(F\). From \eqref{e4.5}, we have that
\[
F(\gamma_2, t\eta_0) = \ell \, |\nabla \Psi_t(x + t\eta_0(x)\nu(x))|.
\]
Differentiating under the norm and evaluating at \(t=0\) gives
\[
\partial_\eta F(\gamma_2,0)[\eta_0] = \nu(x) \cdot \dot{B}(x),
\]
where
\[
\dot{B}(x) = \frac{d}{dt}\Big|_{t=0} \nabla \Psi_t(x + t\eta_0(x)\nu(x)) = \nabla \Psi'(x) + \eta_0(x) D^2 \Psi_\mathrm{tri}(x) \nu(x).
\]
Thus
\[
\partial_\eta F(\gamma_2,0)[\eta_0] = \partial_\nu \Psi' + \partial_{\nu\nu} \Psi_\mathrm{tri} \, \eta_0,
\]
and substituting back into the expression for \(\partial_\eta H\) yields \eqref{e4.8}.
\end{proof}

To determine the spectrum of \(\partial_\eta H(\gamma_2,0)\), we expand the shape derivative \(\Psi'\) in spherical harmonics. Using \eqref{e2.4} and \eqref{e4.7}, we can get the following result by direct computation.

\begin{proposition}\label{pro-4.4}
Let \(\eta_0 \in X\) admit the expansion \(\eta_0(\theta) = \sum_{k \geq 1} \alpha_k \cos(k\theta)\). Then the shape derivative \(\Psi'\) solving \eqref{e4.7} is given by
\begin{equation}
\Psi'(r,\theta) :=
\begin{cases}
\Psi'_1(r,\theta) = \sum_{k \geq 1} D_k r^k \alpha_k \cos(k\theta), & r \in [0,\lambda), \\
\Psi'_2(r,\theta) = \sum_{k \geq 1} (E_k r^{-k} + F_k r^k) \alpha_k \cos(k\theta), & r \in [\lambda,1), \nonumber
\end{cases}
\end{equation}
where
\[
D_k = \frac{(\lambda^{-2k}-1)(\gamma_2^2 - \gamma_1 \gamma_2)}{2\gamma_2(\lambda^{-2k}-1) + 2\gamma_1(\lambda^{-2k}+1)} - \frac{\gamma_2}{2},
\]
\[
E_k = \frac{\gamma_2^2 - \gamma_1 \gamma_2}{2\gamma_2(\lambda^{-2k}-1) + 2\gamma_1(\lambda^{-2k}+1)}, \quad
F_k = \frac{\gamma_1 \gamma_2 - \gamma_2^2}{2\gamma_2(\lambda^{-2k}-1) + 2\gamma_1(\lambda^{-2k}+1)} - \frac{\gamma_2}{2}.
\]
\end{proposition}

\begin{proof}
We seek solutions of the separated form
\[
\Psi'(r,\theta) =
\begin{cases}
\sum_{k \geq 1} S_1(r) \alpha_k \cos(k\theta), & r \in [0,\lambda), \\
\sum_{k \geq 1} S_2(r) \alpha_k \cos(k\theta), & r \in [\lambda,1).
\end{cases}
\]
Substituting into \(\Delta \Psi' = 0\) yields the Euler equations
\begin{equation}
\partial_{rr}S_j + \frac{1}{r} \partial_{r}S_j - \frac{k^2}{r^2} S_j = 0, \qquad j=1,2. \nonumber
\end{equation}
The general solutions are \(S_1(r) = C_k r^{-k} + D_k r^k\) and \(S_2(r) = E_k r^{-k} + F_k r^k\). Regularity at the origin forces \(C_k = 0\). The transmission conditions on \(\partial B_\lambda\) together with the outer boundary condition on \(\partial \Omega_0\) yield the linear system
\[
\begin{cases}
E_k + F_k = -\frac{\gamma_2}{2}, \\
D_k \lambda^k = E_k \lambda^{-k} + F_k \lambda^k, \\
\frac{1}{\gamma_1} D_k \lambda^{k-1} = \frac{1}{\gamma_2} (-E_k \lambda^{-k-1} + F_k \lambda^{k-1}).
\end{cases}
\]
Solving this system yields the explicit coefficients \(D_k\), \(E_k\), and \(F_k\) stated above.
\end{proof}

Combining Propositions \ref{pro-4.3} and \ref{pro-4.4} gives the dispersion relation
\begin{equation}
\partial_\eta H(\gamma_2,0)[\eta_0] = 2 \partial_\nu \Psi_\mathrm{tri} \bigl( \partial_\nu \Psi' + \partial_{\nu\nu} \Psi_\mathrm{tri} \, \eta_0 \bigr) = \gamma_2 \sum_{k \geq 1} \mu_k \alpha_k \cos(k\theta), \nonumber
\end{equation}
where
\begin{equation}\label{e4.9}
\mu_k = (-k E_k + k F_k) + \frac{\gamma_2}{2}.
\end{equation}
Since \(Q^{\gamma_2} > 0\) forces \(\gamma_2 \neq 0\), the linearized operator \(\partial_\eta H(\gamma_2,0)\) is degenerate precisely when \(\mu_k(\gamma_2) = 0\) for some \(k \geq 1\).

A direct computation shows that \(\mu_k(\gamma_2) = 0\) if and only if
\[h_k(\gamma_2):= \frac{2k\left(\gamma_1\gamma_2-\gamma_2^2\right)}{2\gamma_2\left(\lambda^{-2k}-1\right)+2\gamma_1\left(\lambda^{-2k}+1\right)}+\frac{(1-k)\gamma_2}{2}=0,\]
whose roots are
\begin{equation}\label{e4.10}
\gamma_{2k}^* = -\frac{(1-k)(\lambda^{-2k} + 1) + 2k}{(1-k)(\lambda^{-2k} - 1) - 2k} \gamma_1.
\end{equation}
Equivalently,
\begin{equation}
f(\lambda,k) := \frac{(k+1) - (k-1)\lambda^{-2k}}{(k+1) + (k-1)\lambda^{-2k}}, \nonumber
\end{equation}
so that \(\gamma_{2k}^* = f(\lambda,k) \gamma_1\). One readily verifies that \(f(\lambda,1) = 1\) and \(f(\lambda,k) < 1\) for \(k \geq 2\). Moreover,
\[
\frac{df(\lambda,k)}{dk} = -\frac{4\lambda^{-2k} \bigl(1 - (k^2 - 1)\ln\lambda\bigr)}{\bigl(k+1 + (k-1)\lambda^{-2k}\bigr)^2} < 0
\]
for all \(\lambda \in (0,1)\) and \(k \in \mathbb{N}^+ \setminus \{1\}\). Thus each \(\gamma_{2k}^*\) is a simple bifurcation point for all $k\in\mathbb{N}^+$ and \(\lambda \in (0,1)\).

In the remainder of the paper we focus on the case \(k=1\), for which
\begin{equation}
\gamma_2^* = \gamma_{21}^* = \gamma_1 \neq 0. \nonumber
\end{equation}

\subsection{Proof of Theorem \ref{thm1.2}}
In this subsection, we shall give the proof to Theorem \ref{thm1.2}.
\begin{proof}[Proof of Theorem \ref{thm1.2}]
We first prove statements (1) and (2). From the preceding analysis,
\begin{equation}
H(\gamma_2,0) = 0 \nonumber
\end{equation}
holds for all \(\gamma_2 \in \mathbb{R} \setminus \{0\}\), so condition (H1) of the Crandall--Rabinowitz theorem \ref{thm6.1} is satisfied.

At \(\gamma_2 = \gamma_1\), the kernel \(\mathcal{N}(\partial_\eta H(\gamma_1,0))\) is one-dimensional and spanned by \(\eta^*(\theta) = \alpha_1 \cos\theta \in X\), while the range \(\mathcal{R}(\partial_\eta H(\gamma_1,0))\) consists of all \(\phi \in Y\) orthogonal to \(\eta^*\) in \(L^2(\partial \Omega_0)\). Moreover,
\begin{equation}
\partial_{\gamma_2 \eta} H(\gamma_1,0)[1,\eta^*] = -\frac{\gamma_1 \lambda^2}{2} \eta^* \notin \mathcal{R}(\partial_\eta H(\gamma_1,0)). \nonumber
\end{equation}
Since \(\gamma_1 \neq 0\) and \(\lambda \in (0,1)\), the coefficient \(-\gamma_1 \lambda^2 / 2 \neq 0\). Thus condition (H2) holds, and the Crandall--Rabinowitz theorem yields the desired local bifurcation.

It remains to prove statement (3). Let \(\eta \in C^{2,\alpha}(\partial \Omega_0)\) be a small perturbation. Then the perturbed boundary \(\partial \Omega_\eta\) is of class \(C^{2,\alpha}\). Since the inner boundary \(\partial B_\lambda\) is smooth, Schauder estimates \cite[Theorems 6.8 and 6.9]{Gilbarg} (or \cite[Lemma 2.1]{FallMW}) imply that the outer solution \(\psi_{2\eta}\) belongs to \(C^{2,\alpha}\) up to \(\partial \Omega_\eta\). By continuity, \(\Delta \psi_{2\eta} = \gamma_2\) holds on \(\partial \Omega_\eta\). Applying \cite[Proposition 5.4.12]{HenrotP} gives
\begin{equation}\label{e4.11}
\gamma_2 = \partial_{\nu\nu} \psi_2 + H \partial_\nu \psi_2 + \Delta_\tau \psi_2 \quad \text{on}\quad \partial \Omega_\eta,
\end{equation}
where \(H\) is the mean curvature of \(\partial \Omega_\eta\) and \(\Delta_\tau\) is the Laplace--Beltrami operator. Substituting the boundary conditions on $\partial\Omega$ into \eqref{e4.11} yields
\[
H = \ell \frac{\gamma_2 - m}{\sqrt{Q}} \quad \text{on}\quad \partial \Omega_\eta,
\]
which is constant. Aleksandrov's theorem \ref{thm6.4} then implies that \(\Omega_\eta\) is a ball, completing the proof.
\end{proof}

The proof of Corollary \ref{co1} follows the same lines as that of Corollary \ref{co} and is therefore omitted.

\medskip

\section{The two-phase overdetermined problem (\ref{e1.4}) with surface tension}

In this section, we construct nontrivial admissible domains for the two-phase overdetermined problem \eqref{e1.4} with surface tension.
For the exterior problem, condition
\[
|\nabla\widetilde{\psi}_{2}(x)|\to0
\qquad\text{as }|x|\to\infty
\]
does not determine a unique radial solution. Instead, the radial solutions constitute a one-parameter family
\[
\widetilde{\psi}_{2}(x)=a\ln|x|+O(1),
\]
parameterized by the logarithmic coefficient $a$. In the present section, we fix $a\neq0$ and study local bifurcation from the corresponding radial solution. Therefore, $a$ is viewed as a given trivial solution branch, rather than an additional parameter in the formulation of problem \eqref{e1.4}.
Thus, let us fix a nonzero constant $a\in\mathbb{R}$ satisfying
\begin{equation}\label{ezy}
a^2<\frac{3\beta}{2q}
\end{equation}
throughout this section. By standard elliptic theory for the interior Dirichlet problem and the classical theory of the exterior Green function with pole at infinity, for every sufficiently small $\eta\in \widetilde{X}$ defined in \eqref{e2.2}, there exists a unique pair
$$(\widetilde{\psi}_{1\eta},\widetilde{\psi}_{2\eta})\in C_{even}^{2,\alpha}(\overline{\Omega_{\eta}})\times C_{even,loc}^{2,\alpha}(\overline{\mathbb{R}^2\setminus\overline{\Omega_{\eta}}})$$
to the following two-phase transmission problem
\begin{equation}\label{e5.1}
\begin{cases}
\Delta \widetilde{\psi}_1 = \widetilde{\gamma} & \text{in } \Omega_{\eta}, \\
\Delta \widetilde{\psi}_2 = 0 & \text{in } \mathbb{R}^2 \setminus \overline{\Omega_{\eta}}, \\
\widetilde{\psi}_1 = \widetilde{\psi}_2=0 & \text{on } \partial \Omega_{\eta}, \\
 \widetilde{\psi}_2-a\ln|x|=O(1) & \text{for}~ |x|\rightarrow+\infty.
\end{cases}
\end{equation}
We first describe the trivial branch explicitly.

\begin{lemma}\label{lem5.1}
For any positive parameter \(\widetilde{\gamma} \in \mathbb{R}\), there exist a unique constant
\begin{equation}\label{e5.2}
Q(\widetilde{\gamma}) = \frac{\widetilde{\gamma}^2}{4}+\beta-qa^2
\end{equation}
such that on \((\Omega, \mathbb{R}^2\setminus \Omega) = (\Omega_0, \mathbb{R}^2\setminus \Omega_0)\) the problem \eqref{e1.4} admits a unique radially symmetric solution $(\widetilde{\psi}_1,\widetilde{\psi}_2)=(\widetilde{\psi}_{1triv},\widetilde{\psi}_{2triv})$ given by
\begin{equation}\label{e5.3}
\begin{cases}
\widetilde{\psi}_{1triv}(|x|) = -\frac{\left(1-|x|^2\right)\widetilde{\gamma}}{4}, & |x| \in [0,1), \\
\widetilde{\psi}_{2triv}(|x|) = a\ln |x|, & |x| \in [1,+\infty),
\end{cases}
\end{equation}
where the nonzero logarithmic coefficient a has been fixed above.
\end{lemma}

Similarly, we reformulate problem \eqref{e1.4} as an abstract operator equation from the spaces \(\widetilde{X}\) to \(\widetilde{Y}\).
Based on the \(\eta\)-dependent solution $(\widetilde{\psi}_{1\eta},\widetilde{\psi}_{2\eta})$ of \eqref{e5.1}, we introduce the operator equation
\begin{equation}\label{e5.4}
\widetilde{F}(\widetilde{\gamma},\eta) =\Pi_0\left( \big| \partial_{\nu_\eta} \widetilde{\psi}_{1\eta}\big|_{\partial \Omega_0}\big|^2+\beta\mathcal{K}_{\eta}-q\big| \widetilde{\partial}_{\nu_\eta} \psi_{2\eta}\big|_{\partial \Omega_0}\big|^2-Q(\widetilde{\gamma})\right) ,
\end{equation}
where $\Pi_0:C^{0,\alpha}(\partial \Omega_0)\rightarrow M_0$ with $M_0$ given by \eqref{zy} is the projection operator, defined by
$$
 \Pi_0(\phi)=\phi-\frac{1}{|\partial \Omega_0|}\int_{\partial \Omega_0}\phi\, ds
$$
and the boundary trace is understood in the sense
\[
\partial_{\nu_\eta} \widetilde{\psi}_{1\eta}\big|_{\partial \Omega_0}(x) = \nabla \widetilde{\psi}_{1\eta}(x + \eta(x)\nu(x)) \cdot \nu_\eta(x + \eta(x)\nu(x)), \quad x \in \partial \Omega_0,
\]
and
\[
\partial_{\nu_\eta} \widetilde{\psi}_{2\eta}\big|_{\partial \Omega_0}(x) = \nabla \widetilde{\psi}_{2\eta}(x + \eta(x)\nu(x)) \cdot \nu_\eta(x + \eta(x)\nu(x)), \quad x \in \partial \Omega_0,
\]
with $\nu_{\eta}$ being the outward unit normals to $\partial\Omega_{\eta}$, $\mathcal{K}_{\eta}=\frac{(1+\eta)^2+2\eta'^{2}-(1+\eta)\eta''}{\left((1+\eta)^2+\eta'^{2}\right)^{\frac{3}{2}}}$ and $Q(\beta)$ is given by \eqref{e5.2}. Then \(\widetilde{F}(\beta,\eta) = 0\) if and only if the Neumann interface boundary condition in \eqref{e1.4} holds on \(\partial \Omega_\eta\), which is equivalent to \((\widetilde{\psi}_{1\eta}, \widetilde{\psi}_{2\eta})\) solving \eqref{e1.4}.

\subsection{The spectrum of the linearized operator \(\partial_{\eta}\widetilde{F}(\beta,0)\)}

The goal is to find nontrivial \(\eta \in \widetilde{X}\) solving
$$
\widetilde{F}(\beta,\eta) = \Pi_0\left(\big| \partial_{\nu_\eta} \widetilde{\psi}_{1\eta}\big|_{\partial \Omega_0}\big|^2+\beta\mathcal{K}_{\eta}-q\big| \widetilde{\partial}_{\nu_\eta} \psi_{2\eta}\big|_{\partial \Omega_0}\big|^2-Q(\beta)\right)=0.
$$
The Fr\'{e}chet differentiability of \(\widetilde{F}\) near \((\beta,0)\) follows from the differentiability of the solution map \(\eta \mapsto \widetilde{\psi}_{1\eta}\) and \(\eta \mapsto \widetilde{\psi}_{2\eta}\) and their derivatives, using \cite[Theorem 5.3.2]{HenrotP} and elliptic regularity \cite{Gilbarg}. By Lemma \ref{lem5.1}, \(\widetilde{F}(\beta,0) = 0\).

For \(\eta_0 \in \widetilde{X}\), the partial Fr\'{e}chet derivatives coincide with the G\^{a}teaux derivatives
\[
\partial_\eta \widetilde{F}(\beta,0)[\eta_0] = \frac{d}{dt}\Big|_{t=0} \widetilde{F}(\beta,t\eta_0).
\]
Let \((\widetilde{\psi}_{1t},\widetilde{\psi}_{2t}) = (\widetilde{\psi}_{(1t\eta_0)},\widetilde{\psi}_{(2t\eta_0)})\) solve \eqref{e5.1}. The shape derivative \((\widetilde{\psi}'_1,\widetilde{\psi}'_2)\) of \((\widetilde{\psi}_{1t},\widetilde{\psi}_{2t}) \) has the following characterization

\begin{lemma}\label{lem5.2}
Let \(\eta_0 \in \widetilde{X}\) be fixed. Then \((\widetilde{\psi}'_1,\widetilde{\psi}'_2)\) exist and solve
\begin{equation}\label{e5.5}
\begin{cases}
\Delta \widetilde{\psi}'_{1} = 0 & \text{in } \Omega_0, \\
\widetilde{\psi}'_{1} = -\partial_{\nu}\widetilde{\psi}_{1triv} \eta_0 & \text{on } \partial \Omega_0,
\end{cases}
\qquad
\begin{cases}
\Delta \widetilde{\psi}'_{2} = 0 & \text{in } \mathbb{R}^2 \setminus \overline{\Omega_0}, \\
\widetilde{\psi}'_{2} = -\partial_{\nu}\widetilde{\psi}_{2triv} \eta_0 & \text{on } \partial \Omega_0, \\
\nabla\widetilde{\psi}'_{2} \rightarrow0 & \text{as } |x|\rightarrow +\infty,
\end{cases}
\end{equation}
where $\widetilde{\psi}_{1triv}$ and $\widetilde{\psi}_{2triv}$ are trivial solutions given in \eqref{e5.3} and \(\widetilde{\psi}'_{1}\) (resp.\ \(\widetilde{\psi}'_{2}\)) is the shape derivative with respect to the inner (resp.\ outer) perturbation on the interface $\partial\Omega_0$.
\end{lemma}

\begin{proof}
The equations and boundary conditions satisfied by \(\widetilde{\psi}'_{1}\) and \(\widetilde{\psi}'_{2}\) on
$\partial \Omega_0$ coincide with those in Lemma \ref{lem3.2}. The boundary condition at infinity for \(\widetilde{\psi}'_{2}\) is then a direct consequence of the linearity of the shape derivative.
\end{proof}

\begin{proposition}\label{pro-5.3}
The map \(\widetilde{F}\) $($defined in \eqref{e5.4}$)$ is Fr\'{e}chet differentiable near \((\beta,0)\). Moreover, for any \(\eta_0 \in \widetilde{X}\), there holds that
\begin{equation}\label{e5.6}
\begin{aligned}
\partial_{\eta} \widetilde{F}(\beta,0)[\eta_0] =&2\partial_{\nu}\widetilde{\psi}_{1triv}\left(\partial_{\nu}\widetilde{\psi}'_{1}+\partial_{\nu\nu}\widetilde{\psi}_{1triv} \eta_0\right)
-\beta(\partial_{\theta\theta}+1)\eta_0\\
-&2q\partial_{\nu}\widetilde{\psi}_{2triv}\left(\partial_{\nu}\widetilde{\psi}'_{2}+
\partial_{\nu\nu}\widetilde{\psi}_{2triv} \eta_0\right),
\end{aligned}
\end{equation}
where $\widetilde{\psi}_{1triv}$ and $\widetilde{\psi}_{2triv}$ are trivial solutions given in \eqref{e5.3} and \(\widetilde{\psi}'_{1}\) and \(\widetilde{\psi}'_{2}\) solve \eqref{e5.5}, respectively.
\end{proposition}

\begin{proof}
We compute the G\^{a}teaux derivative with respect to \(\eta\):
\[
\partial_\eta \widetilde{F}(\beta,0)[\eta_0] = \frac{d}{dt}\Big|_{t=0} \widetilde{F}(\beta,t\eta_0).
\]
Following the same computational procedure as in the proof of Proposition \ref{pro-3.3}, the first and third terms on the right-hand side of \eqref{e5.6} are readily obtained. We now focus mainly on the linearity of the curvature term. Since the curvature of a planar curve represented as a radial graph is given by
\[
\mathcal{K}_{\eta}
=
\frac{(1+\eta)^2+2(\partial_{\theta}\eta)^2
-(1+\eta)\partial_{\theta\theta}\eta}
{\left((1+\eta)^2+(\partial_{\theta}\eta)^2\right)^{3/2}},
\]
we may expand it around the unit disk. Indeed, if the perturbation $\eta$ is sufficiently small, then
\[
\mathcal{K}_{\eta}
=
\frac{
1+2\eta-\partial_{\theta\theta}\eta
+O\!\left(\eta^2+(\partial_{\theta}\eta)^2\right)
}
{
1+3\eta
+O\!\left(\eta^2+(\partial_{\theta}\eta)^2\right)
},
\]
which yields
\[
\mathcal{K}_{\eta}
=
1-\eta-\partial_{\theta\theta}\eta
+O\!\left(\eta^2+(\partial_{\theta}\eta)^2\right).
\]
Therefore, the linearization of the curvature operator at the unit disk is given by
\[
\partial_{\eta}\mathcal{K}(0)[\eta_0]
=
-\left(\partial_{\theta\theta}+1\right)\eta_0.
\]
\end{proof}

To analyze the spectrum of \(\partial_{\eta} \widetilde{F}(\beta,0)\),
we employ the Fourier expansion \eqref{e2.5} together with the method of separation of variables, which yields the following explicit characterization.

\begin{proposition}\label{pro-5.4}
Let \(\eta_0 \in \widetilde{X}\) with the following Fourier expansion
\[
\eta_0(1,\theta) = \sum_{k \geq 2} \alpha_k \cos(k\theta)
\]
for $k\in N^+_{even}$, where $N^+_{even}$ is the set of positive even integers and \(\alpha_k \) are real constants given in \eqref{e2.5}. Then the shape derivative \(\widetilde{\psi}'_{1}\) and \(\widetilde{\psi}'_{2}\) solve \eqref{e5.5} are
$$
\widetilde{\psi}'_{1}(r,\theta) = -\sum_{k \geq 2}\frac{\widetilde{\gamma}}{2}r^k \alpha_k\cos(k\theta),
$$
and
$$
\widetilde{\psi}'_{2}(r,\theta) =- \sum_{k \geq 2}ar^{-k} \alpha_k\cos(k\theta).
$$
\end{proposition}

\begin{proof}
We seek separated solutions
\[\widetilde{\psi}'_{1}(r,\theta) = \sum_{k=2}^\infty \widetilde{S}_{1}(r) \alpha_k \cos(k\theta)\]
and
\[\widetilde{\psi}'_{2}(r,\theta) = \sum_{k=2}^\infty \widetilde{S}_{2}(r) \alpha_k \cos(k\theta).\]
Substituting into Laplace's equation yields the Euler ODEs
\[
\partial_{rr}\widetilde{S} + \frac{1}{r} \partial_r\widetilde{S} - \frac{k^2}{r^2}\widetilde{ S} = 0.
\]
The general solutions are \(\widetilde{S}_{1}(r) = \mathcal{A}_k r^{-k} + \mathcal{B}_k r^k\) and \(\widetilde{S}_{2}(r) = \mathcal{C}_k r^{-k} + \mathcal{D}_k r^k\). Applying the boundary conditions in \eqref{e5.5} produces the linear systems whose solutions are the coefficients listed as follows:
\[
\mathcal{A}_k = 0, \qquad
\mathcal{B}_k = -\frac{\widetilde{\gamma}}{2},
\]
\[
\mathcal{C}_k = -a, \qquad
\mathcal{D}_k = 0.
\]
Then we finish the proof.
\end{proof}

By Propositions \ref{pro-5.3} and \ref{pro-5.4}, we can obtain the following dispersion relation
\begin{equation}\label{e5.7}
\partial_\eta \widetilde{F}(\widetilde{\gamma}, 0)[\eta_0] = \sum_{k \geq 2} \Sigma_k \alpha_k \cos(k\theta),
\end{equation}
with
\begin{equation}\label{e5.8}
\Sigma_k (\widetilde{\gamma})= (1-k)\left(\frac{\widetilde{\gamma}^2}{2}+2qa^2-\beta(k+1)\right)
\end{equation}
for $k\in N^+_{even}$.

A direct computation shows that \(\Sigma_k(\widetilde{\gamma}) = 0\) if and only if
\begin{equation}\label{e5.9}
\widetilde{\gamma}_{k,\pm}^* =\pm\sqrt{2\beta(k+1)-4qa^2} .
\end{equation}
It follows from \eqref{ezy} that the bifurcation points $\widetilde{\gamma}_{k,\pm}^*$ defined in \eqref{e5.9} are well-defined and simple for every $k\in\mathbb{N}^{+}{\mathrm{even}}$.

\subsection{Proof of Theorem \ref{thm1.3}}

We complete the proof of Theorem \ref{thm1.3} by applying the Crandall--Rabinowitz local bifurcation theorem (Theorem \ref{thm6.1} in the Appendix).

\begin{proof}[Proof of Theorem \ref{thm1.3}]
From the preceding analysis,
$$
\widetilde{F}(\widetilde{\gamma},0)=0
$$
holds for all \(\widetilde{\gamma} \in \mathbb{R}\), where \(\widetilde{F}\) is the operator defined in \eqref{e5.5}. Thus condition (H1) of Theorem \ref{thm6.1} is satisfied.

It remains to verify condition (H2) at the critical value \(\widetilde{\gamma} = \widetilde{\gamma}_{k,\pm}^*\) given in \eqref{e5.9}. From the dispersion relation \eqref{e5.7}--\eqref{e5.8}, the kernel \(\mathcal{N}(\partial_\eta \widetilde{F}(\widetilde{\gamma}_{k,\pm}^*,0))\) is one-dimensional and spanned by
\[
\widetilde{\eta}^*(\theta) = \alpha_k \cos(k\theta) \in X.
\]
The range \(\mathcal{R}(\partial_\eta \widetilde{F}(\widetilde{\gamma}_{k,\pm}^*,0))\) is the closed subspace of \(\widetilde{Y}\) consisting of all \(\phi \in \widetilde{Y}\) such that
\[
\int_0^{2\pi} \phi(\theta)\cos(k\theta) \, d\theta = 0.
\]
Consequently, \(\widetilde{Y} / \mathcal{R}(\partial_\eta \widetilde{F}(\widetilde{\gamma}_{k,\pm}^*,0))\) is one-dimensional and generated by \(\widetilde{\eta}^*\).

Furthermore, a direct computation using \eqref{e5.7}--\eqref{e5.8} yields
\begin{equation}
\begin{aligned}
\partial_{\gamma\eta} \widetilde{F}(\widetilde{\gamma}_{k,\pm}^*,0)[1,\widetilde{\eta}^*] &= (1-k)\widetilde{\gamma}_{k,\pm}^* \widetilde{\eta}^* =\pm(1-k)\sqrt{2\beta(k+1)-4qa^2}~ \widetilde{\eta}^*, \nonumber
\end{aligned}
\end{equation}
which is obviously not in $\mathcal{R}(\partial_\eta \widetilde{F}(\widetilde{\gamma}_{k,\pm}^*,0)).$
So condition (H2) holds.
Applying the Crandall--Rabinowitz bifurcation theorem \ref{thm6.1} completes the proof of Theorem \ref{thm1.3}.
\end{proof}

\subsection{Proof of Theorem \ref{thm1.4}}
Now let us prove the Theorem \ref{thm1.4}.
\begin{proof}
We use the convention
\[ \Delta u = \partial_{\nu\nu}u + \mathcal{K}\partial_\nu u + \Delta_{\tau}u \qquad\text{on }\partial\Omega, \]
where $\nu$ is the unit normal pointing outward from $\Omega$ and $\Delta_\tau$ denotes the tangential Laplacian. Since
\[ \widetilde{\psi}_{1}=\widetilde{\psi}_{2}=0 \qquad\text{on }\partial\Omega, \]
 both functions are constant along the interface. Hence their tangential derivatives vanish and
 \[ \Delta_\tau\widetilde{\psi}_{1} = \Delta_\tau\widetilde{\psi}_{2} =0 \qquad\text{on }\partial\Omega. \]
 Applying the boundary decomposition of the Laplacian to $\widetilde{\psi}_{1}$, we obtain
 \[ \widetilde{\gamma} = \partial_{\nu\nu}\widetilde{\psi}_{1} + \mathcal{K}\partial_\nu\widetilde{\psi}_{1} = m_1+\mathcal{K}\partial_\nu\widetilde{\psi}_{1}. \] Therefore,
 \begin{equation}\label{e5.11}
 \mathcal{K}\partial_\nu\widetilde{\psi}_{1} =  \widetilde{\gamma} -m_1.
 \end{equation}
 Similarly, since $\Delta\widetilde{\psi}_{2}=0$ in the exterior domain,
 \[ 0 = \partial_{\nu\nu}\widetilde{\psi}_{2} + \mathcal{K}\partial_\nu\widetilde{\psi}_{2} = m_2+\mathcal{K}\partial_\nu\widetilde{\psi}_{2}, \]
 and hence
 \begin{equation}\label{e5.12}
  \mathcal{K}\partial_\nu\widetilde{\psi}_{2} = -m_2.
   \end{equation}
 Because both functions are constant on $\partial\Omega$, their gradients are normal to the boundary. Thus
 \[ |\nabla\widetilde{\psi}_{1}|^2 = |\partial_\nu\widetilde{\psi}_{1}|^2, \qquad |\nabla\widetilde{\psi}_{2}|^2 = |\partial_\nu\widetilde{\psi}_{2}|^2 \qquad\text{on }\partial\Omega. \]
 Multiplying the Bernoulli condition by $\mathcal{K}^2$ and using
 \eqref{e5.11}--\eqref{e5.12}, we find
 \[ ( \widetilde{\gamma} -m_1)^2 + \beta \mathcal{K}^3 -qm_2^2 = Q\mathcal{K}^2 \qquad\text{on }\partial\Omega. \]
 Equivalently,
 \[ P(\mathcal{K})=0 \qquad\text{on }\partial\Omega, \]
 where
 \[ P(t) := \beta t^3 - Qt^2 + ( \widetilde{\gamma} -m_1)^2 -qm_2^2. \]
 Since $\beta>0$, the polynomial $P$ is not identically zero and therefore has only finitely many real roots.
 On the other hand, $\mathcal{K}$ is continuous and $\partial\Omega$ is connected, so $\mathcal{K}(\partial\Omega)$ is a connected subset of $\mathbb R$. Since \[ \mathcal{K}(\partial\Omega) \subset \{t\in\mathbb R:P(t)=0\}, \]
 and the latter set is finite, it follows that $\mathcal{K}(\partial\Omega)$ consists of a single point. Therefore
 \[ \mathcal{K}\equiv \mathcal{K}_0 \qquad\text{on }\partial\Omega \]
 for some constant $\mathcal{K}_0$. It follows from Aleksandrov's theorem \ref{thm6.4} that $\partial\Omega$ is a circle and $\Omega$ is a disk. Let $x_0$ denote the center of this disk.
 The interior Dirichlet problem
 \[ \Delta\widetilde{\psi}_{1}= \widetilde{\gamma} \quad\text{in }\Omega, \qquad \widetilde{\psi}_{1}=0 \quad\text{on }\partial\Omega \]
 has a unique solution. Since the problem is invariant under rotations about $x_0$, uniqueness implies that $\widetilde{\psi}_{1}$ is radial with respect to $x_0$. Likewise, after fixing the logarithmic coefficient of the exterior solution at infinity, the exterior Dirichlet problem is unique. Rotational invariance about $x_0$ then implies that $\widetilde{\psi}_{2}$ is radial with respect to the same center.
\end{proof}

\bigskip

\begin{center}
\textbf{Acknowledgement}
\end{center}

The research of C.F. Gui is supported by
University of Macau research grants CPG2024-00016-FST,
CPG202500032-FST, SRG2023-00011-FST, MYRG-GRG2023-00139-FST-UMDF,
UMDF Professorial Fellowship of Mathematics, Macao SAR
FDCT0003/2023/RIA1 and Macao SAR FDCT0024/2023/RIB1, NSFC No.
12531010. The research of J. Wang is partially supported by National
Key R$\&$D Program of China 2022YFA1005601 and National Natural
Science Foundation of China 12371114. The research of W. Yang is
partially supported by National Key R\&D Program of China
2022YFA1006800, NSFC No. 12531010, No. 12171456 and No. 12271369,
FDCT No. 0070/2024/RIA1, UMDF No. TISF/2025/006/FST, No.
MYRG-GRG2025-00051-FST, No. MYRG-GRG2024-00082-FST-UMDF and Startup
Research Grant No. SRG2023-00067-FST. The research of Y. Zhang is
partially supported by National Natural Science Foundation of China
No. 12301133 and the Postdoctoral Science Foundation of China (No.
2023M741441, No. 2024T170353) and Jiangsu Education Department (No.
23KJB110007).

\section{Appendix}

This appendix collects the Crandall--Rabinowitz local bifurcation theorem, Reichel's theorem, Aleksandrov's theorem, and the implicit function theorem for reference.

\begin{theorem}[Crandall--Rabinowitz \cite{CrandallR}] \label{thm6.1}
Let \(X\) and \(Y\) be Banach spaces and let \(F \colon \mathbb{R} \times X \to Y\) be a \(C^k\) map with \(k \geq 2\). Denote by \(\mathcal{N}(T)\) and \(\mathcal{R}(T)\) the kernel and range of a linear operator \(T\). Suppose that
\begin{itemize}
\item[(H1)] \(F(\lambda, 0) = 0\) for all \(\lambda \in \mathbb{R}\);
\item[(H2)] there exists \(\lambda_* \in \mathbb{R}\) such that \(F_x(\lambda_*, 0)\) is Fredholm, both \(\mathcal{N}(F_x(\lambda_*, 0))\) and \(Y / \mathcal{R}(F_x(\lambda_*, 0))\) are one-dimensional, the kernel is generated by some \(x_* \in X\), and the transversality condition
\[
F_{\lambda x}(\lambda_*, 0)[1, x_*] \notin \mathcal{R}(F_x(\lambda_*, 0))
\]
holds.
\end{itemize}
Then \(\lambda_*\) is a bifurcation point: there exist \(\varepsilon > 0\) and a \(C^{k-1}\) curve
\[
\bigl\{ (\lambda, x) = \bigl( \Lambda(s), s \psi(s) \bigr) : |s| < \varepsilon \bigr\} \subset \mathbb{R} \times X
\]
with \(F(\lambda, x) = 0\), \(\Lambda(0) = \lambda_*\), \(\psi(0) = x_*\), and the maps \(s \mapsto \Lambda(s)\) and \(s \mapsto s\psi(s)\) are of class \(C^{k-1}\).
\end{theorem}

\begin{definition}
Let \(\Omega\) and \(D\) be two simply connected domains of class \(C^2\) with \(\overline{D} \subset \Omega\). Then \(\Omega \setminus \overline{D}\) is called a ring-shaped domain.
\end{definition}

Consider a solution \(u \in C^2(\overline{\Omega \setminus D})\) of the boundary-value problem
\begin{equation}\label{e6.1}
\begin{cases}
\Delta u + f(u) = 0 & \text{in } \Omega \setminus \overline{D}, \\
u = 0 & \text{on } \partial \Omega, \\
u = a > 0 & \text{on } \partial D, \\
0 < u(x) < a & \text{in } \Omega \setminus \overline{D},
\end{cases}
\end{equation}
subject to one of the following Neumann conditions:
\begin{subequations}\label{e6.2}
\begin{align}
\frac{\partial u}{\partial\nu}|_{\partial D}=c_1 &\quad\text{ if $\Omega$ is a ball}, \label{e6.2a}\\
\frac{\partial u}{\partial\nu}|_{\partial \Omega}=c_0 &\quad\text{ if $D$ is a ball },\label{e6.2b}\\
\frac{\partial u}{\partial\nu}|_{\partial D}=c_1, \quad \frac{\partial u}{\partial\nu}|_{\partial \Omega}=c_0,  &\quad\text{ if $\Omega\setminus D$ is a general ring domain } ,\label{e6.2c}
\end{align}
\end{subequations}
where \(\nu\) denotes the unit inner normal with respect to \(\Omega \setminus D\).

\begin{theorem}[Reichel \cite{Reichel}]\label{thm6.3}
For the boundary-value problem \eqref{e6.1} with \(f = f_1 + f_2 \colon [0,a] \to \mathbb{R}\), where \(f_1\) Lipschitz continuous and \(f_2\) increasing and satisfies one of the three boundary conditions \eqref{e6.2a}–\eqref{e6.2c}, then the domain \(\Omega \setminus D\) is an annulus and every solution is radially symmetric and  decreasing in $r$.
\end{theorem}

\begin{theorem}[Aleksandrov \cite{Aleksandrov}]\label{thm6.4}
A compact embedded hypersurface in \(\mathbb{R}^N\) with constant mean curvature must be a sphere.
\end{theorem}

\begin{theorem}[Implicit function theorem]\label{thm6.5}
Let \(X_1\), \(X_2\), and \(Y\) be Banach spaces and let \(F \in C^k(U_1 \times U_2, Y)\) with \(k \geq 1\), where \(U_1 \times U_2\) is open in \(X_1 \times X_2\). Suppose \(F(x_1^*, x_2^*) = 0\) and that \(\partial_{x_2} F(x_1^*, x_2^*)\) is a bounded invertible linear operator from \(X_2\) onto \(Y\). Then there exist neighborhoods \(W_1\) of \(x_1^*\) in \(U_1\) and \(W_2\) of \(x_2^*\) in \(U_2\), together with a map \(g \in C^k(W_1, W_2)\), such that
\begin{enumerate}
\item \(F(x_1, g(x_1)) = 0\) for all \(x_1 \in W_1\);
\item if \(F(x_1, x_2) = 0\) for some \((x_1, x_2) \in W_1 \times W_2\), then \(x_2 = g(x_1)\).
\end{enumerate}
\end{theorem}

\end{document}